%% file: ExtrapolBSD-arxiv.tex
\documentclass[11pt]{amsart}

\usepackage{geometry}                
\usepackage{graphicx}
\usepackage{amssymb}
\usepackage{enumerate}
\usepackage{dsfont}
\usepackage[utf8]{inputenc}
\usepackage[T1]{fontenc}
\usepackage[french,english]{babel}
\usepackage{hyperref}
\usepackage{color}
\usepackage{mathrsfs} 
\usepackage{algorithm}
\usepackage{algorithmic}
\usepackage{bm}
\usepackage{amsthm}
\usepackage{mathabx}

\newtheorem{Theorem}{Theorem}[section]
\newtheorem{Definition}[Theorem]{Definition}
\newtheorem{Proposition}[Theorem]{Proposition}

\newtheorem{Assumption}[Theorem]{Assumption}
\newtheorem{Lemma}[Theorem]{Lemma}
\newtheorem{Corollary}[Theorem]{Corollary}
\newtheorem{Remark}[Theorem]{Remark}

\newcommand{\one}[1]{\ensuremath \mathbf{1}_{\{#1\}}}

\newcommand{\NN}{\ensuremath \mathbb{N}}

\newcommand{\QQ}{\ensuremath \mathbb{Q}}

\newcommand{\mc}[1]{\ensuremath {\mathcal{#1}}}

\newcommand{\bb}{\bar{b}}

\newcommand{\tQ}{\tilde{\mathbb{Q}}}

\newcommand{\hX}{\hat{X}}

\newcommand{\hZ}{\hat{Z}}

\newcommand{\hOmega}{\hat{\Omega}}

\newcommand{\homega}{\hat{\omega}}

\newcommand{\hu}{\hat{u}}
\newcommand{\hv}{\hat{v}}

\newcommand{\hQ}{\hat{\mathbb{Q}}}

\newcommand{\supp}{\mathrm{supp}}

\newcommand{\mpar}[1]{\left( #1 \right)}
\newcommand{\mbra}[1]{\left[ #1 \right]}
\newcommand{\mbrace}[1]{\left\{ #1 \right\}}

\newcommand{\spop}[2]{ \pi_{#1}^{#2} } 
\newcommand{\sprse}[4]{\spop{\mc{V}_{#1}}{#2}\!\left[#3\right]\!\left( #4 \right) }

\input{basecom}

\title{ Cubature method to solve BSDEs: error expansion and complexity control
}
\author[Jean-Francois Chassagneux]{%
Jean-Francois Chassagneux$^{*,\S}$}
\author[Camilo A. Garcia Trillos]{Camilo A. Garcia Trillos$^{\dagger\ddagger,\S}$}

\begin{document}

\begin{abstract}
 We obtain an explicit error expansion for the solution of Backward Stochastic Differential Equations (BSDEs) using the cubature on Wiener spaces method. The result is proved under a mild strengthening of the assumptions needed for the application of the cubature method. The explicit expansion can then be used to construct implementable higher order approximations via Richardson-Romberg extrapolation. To allow for an effective efficiency improvement of the interpolated algorithm, we introduce an additional projection on finite grids through interpolation operators. We study the resulting complexity reduction in the case of the linear interpolation. 
\end{abstract}

\maketitle

\renewcommand{\thefootnote}{\fnsymbol{footnote}}
 \footnotetext[1]{Laboratoire de Probabilit\'es, Statistique et Modélisation, Universit\'e Paris Diderot. {\sf chassagneux@math.univ-paris-diderot.fr}}
 \footnotetext[2]{Department of Mathematics, University College London. {\sf camilo.garcia@ucl.ac.uk}. }
 \footnotetext[3]{The author's research  was partially supported by OpenGamma Ltd.}
  \footnotetext[4]{Acknowledgments: The authors would like to thank Christoph Reisinger and Yufei Zhang for pointing out a mistake in a preliminary version of this paper.}

 \renewcommand{\thefootnote}{\arabic{footnote}}


\section{Introduction}
\label{Sec:Introduction}

Let $(\Omega, \P, \cF, \F )$ with $\F=(\cF_t)_{t\in \R^+}$ be a filtered probability space satisfying the usual conditions.  
For some $T>0$, we consider the solution of the Markovian Backward Stochastic Differential Equation
\begin{align}
X_t= & x_0 + \int_0^t b(s,X_s) \ud s + \int_0^t \sigma(s,X_s) \ud W_s, \label{Eq:Fwd}\\
Y_t= & g(X_T) + \int_t^T f(s,X_s,Y_s,Z_s) \ud s + \int_t^T Z_s \ud W_s, \label{Eq:Bwd}
\end{align}
where $W$ is an  $\F$-Brownian Motion taking values in  $\R^r$, $X, Z$ are $\R^d$ processes adapted to $\F$, $Y$ is an $\F$-adapted process valued in $\R$, and $b: \R_+\times \R^d \rightarrow \R^d$, $\sigma: \R_+ \times \R^d \rightarrow  M (d,r)$, $f:\R_+\times \R^d \times \R \times \R^d \rightarrow \R$ and  $g:\R^d \rightarrow \R$  are Lipschitz function. 
In the sequel, we shall impose further regularity conditions for our theoretical analysis, see Assumptions \ref{H0} or \ref{H1} below.

 An important property of the solution of a Markovian BSDE is that it can be represented as $Y_t = u(t,X_t)$ and $Z_t =v(t,X_t)$ for suitable functions $u,v$ satisfying, at least in a viscosity sense, the PDE
 \begin{align}\label{eq de pde}
 \begin{cases}
 \partial_t u + \cL u + f(\cdot,u,v) = 0 & \text{ on } [0,T)\times \R^d
 \\
 u(T,\cdot) = g(\cdot) &
 \end{cases}
 ,
 \end{align}
 where $\cL$ is the Dynkin operator associated to the diffusion $X$.
 Moreover, when sufficient regularity is available, we have that $v(t,X_t) := \partial_x u(t,X_t) \sigma(t,X_t)$.

 Approximating $(Y,Z)$ allows then to solve numerically and in a probabilistic way, the corresponding PDE for $u$. 
 This has motivated in the past fifteen years an important literature on numerical methods for BSDEs. 
 The main method to approximate \eqref{Eq:Bwd} is a backward programming algorithm based on an Euler scheme, that has been introduced in \cite{briand_donsker-type_2001}  and \cite{bouchard_discrete-time_2004, zhang_numerical_2004}, see the references therein for early works. 
 Since then, many extensions have been considered: high order schemes e.g. \cite{chassagneux_linear_2014,chassagneux_rungekutta_2014}, schemes for reflected BSDEs \cite{bally_quantization_2003,chassagneux_rate_2016}, for fully coupled BSDEs \cite{delarue_forward-backward_2006-1, bender_time_2008}, for quadratic BSDEs \cite{chassagneux_numerical_2016} or McKean-Vlasov BSDEs \cite{chaudru_de_raynal_cubature_2015,chassagneux_numerical_2017}. 
 It is also important to mention that, quite recently, very promising probabilistic forward methods have been introduced to approximate \eqref{Eq:Bwd}  \cite{briand_simulation_2014} or directly the non-linear parabolic PDE \eqref{eq de pde} \cite{henry-labordere_numerical_2014-1}.
The backward algorithm approximating \eqref{Eq:Bwd} requires, to be fully implementable, a good approximation of \eqref{Eq:Fwd} and its associated conditional expectation operator. 
Various methods have been developed, see e.g. 
\cite{crisan_monte_2010,pages_improved_2015,gobet_regression-based_2005}
 and we will focus here on the cubature on Wiener spaces introduced in \cite{lyons_cubature_2004}. 
Broadly speaking, the method considers the space of continuous paths in $\R^d$ on the interval $[0,T]$ ($C([0,T], \R^d)$) to define a finite probability $\Q$ that approximates the Wiener law $\P$. 
As we explain briefly in Section \ref{Sec:Cubature}, see  a full account in \cite{lyons_cubature_2004}, this approximation is chosen to match the expectation of iterated integrals. 

The paper  \cite{crisan_solving_2012} pioneered the use of the cubature method to solve BSDEs. 
Essentially, the algorithm estimates the value of the  field $u$ on the points on the support of the cubature approximating law,  thus giving an approximation scheme for the solution of the BSDE \eqref{Eq:Bwd}.
By its nature, this algorithm can be easily used to implement second order discretization schemes as in  \cite{chassagneux_linear_2014,crisan_second_2014}, and applied in the context of McKean-Vlasov BSDEs as in \cite{chaudru_de_raynal_cubature_2015}. 
The cubature algorithm has been studied under a set of assumptions that guarantee sufficient regularity for the  field (for example, smooth coefficients for the forward equation and the generator of the backward equation, Lipschitz regularity on the boundary condition plus  a structural condition of the type UFG, see Section \ref{Subsec:Main Assumptions} below). 

In this work, we want to study acceleration methods for the Euler approximation of BSDEs, of the same kind as those proved in the linear case $(f=0)$ \cite{talay_expansion_1990} (see also  \cite{gobet_error_2007} for the study of the discrete-time error only in the non-linear case). 
We show that under a very mild strengthening  of the assumptions,  there exists an explicit error expansion for the weak approximation of the BSDE  system given by equations \eqref{Eq:Fwd} and \eqref{Eq:Bwd} using the cubature on Wiener spaces method. 
The explicit expansion exposes the dependence of the error approximation with respect to the general features of the coefficients of the system and the test function. 
Moreover, it opens the possibility to increase the rate of  convergence by using Richardson-Romberg extrapolation techniques. 
However, to effectively improve the efficiency of the algorithm, we need to analyse and improve the complexity growth of the approximation technique. In this work, we consider a technique based on projecting on a finite grid. 


We now present the setup of our work and the numerical schemes
we study in the sequel, and give an overview of the main results of the paper.

\subsection{Main Assumptions}
\label{Subsec:Main Assumptions}

Let us rewrite the process $X$ in \eqref{Eq:Fwd} in its Stratonovich form, that is, let
\[ X_t = x_0 + \int_0^t \bar{b}(s,X_s) ds + \int_0^t \sigma(s,X_s) \circ \ud W_s   \]
where $\bb:\R_+\times \R^d\rightarrow \R^d$ has  i-th component  defined by
\begin{equation}
\bb_{i}(t,x) = b_{i}(t,x) -\frac{1}{2} \sum_{j=1}^r \sum_{k=1}^d  \sigma_{k,j}(t,x) \partial_{x_k} \sigma_{i,j}(t,x).  \label{Eq:b_stratonovich}
\end{equation}

We work under two sufficient assumptions  to guarantee the regularity needed for the cubature method to be effective.

\begin{Assumption}
Let  $M>0$.
\begin{itemize}
\item On the forward coefficients
\begin{enumerate}[i.]
\item $\bar{b}, \sigma_{.,j} \in  C_b^{M+1}(\R^+\times\R^d,\R^d)$ for all $j= 1, \ldots, d$;
\end{enumerate}
\item On the backward coefficients
\begin{enumerate}[i.]
\item $g \in C^{M+1}_b (\R^d, \R)$
\item $f \in C^{M+1}_b (\R^+ \times \R^d \times \R \times \R^d, \R)$.
\end{enumerate}
\end{itemize}
\label{H0}
\end{Assumption}

\begin{Assumption}
Let  $M>0$. There exist $\alpha>0$, such that
\begin{itemize}
\item On the forward coefficients:
\begin{enumerate}[i.]
\item $\bar{b} \in C_b^{M+\alpha+1}(\R^+\times\R^d,\R^d)$ and $\sigma_{.,j} \in  C_b^{M+\alpha+2}(\R^+\times\R^d,\R^d)$ for all $j= 1, \ldots, d$;
\item UFG condition of order $\alpha$ (see Definition 1.1. in \cite{crisan_sharp_2012})
\end{enumerate}
\item On the backward coefficients
\begin{enumerate}[i.]
\item $g$ is Lipschitz continuous.
\item $f \in C^{M}_b (\R^+ \times \R^d \times \R \times \R^d, \R)$
\end{enumerate}
\end{itemize}
\label{H1}
\end{Assumption}

The existence of a classical solution to the PDE \eqref{eq de pde} is assured under either of Assumptions \ref{H0} or \ref{H1}.

\subsection{Forward scheme}

We define a stochastic process on $C_{bv}^{0}([0,T],\R^d) $ -- the space of continuous functions with bounded variation -- as follows.  Let $\hX: \hat{\Omega}  \rightarrow  C_{bv}^{0}([0,T],\R^d) $ be given as the solution of
\begin{equation}
 \hX_t(\homega) := x_0 +  \int_0^t \bb(s,\hX_s(\homega))ds + \int_0^t \sigma (s, \hX_s(\homega) )   \ud \homega_s. 
 \label{Def:XCub}
 \end{equation}
 The integrals in the previous definition are taken in the Riemann-Stieltjes sense, which is possible since we have assumed that the paths $\homega$ are of bounded variation.  
 Let us also define a conditioned form of this stochastic process, given by
\begin{equation}
 \hX_t^{s,x}(\homega) := x +  \int_s^t \bb(r,\hX_r^{s,x}(\homega) )dr + \int_s^t \sigma (r, \hX_r^{s,x}(\homega) )   \ud \homega_{r}. 
 \label{Def:XCubCond}
 \end{equation}

We take as weak approximation of the process $X$ the process $\hat{X}$ under a finite cubature measure $\hat{\Q}$ (see the precise definition in Section \ref{Sec:Cubature} below). 
In other words, we consider a random process in a finite space obtained by solving a finite number of ODEs. For practical implementation, the resulting discrete measure is built as a tree. 
This allows an easy computation of conditional expectations, a property that  is of paramount importance to solve the Backward component that we introduce below.

\vspace{5pt} The precision of approximation provided by the cubature method is given in terms of its order:  it quantifies the degree of iterated integrals that can be  perfectly computed in expectation under the cubature measure. 
Roughly speaking, this is analogous to the maximal degree of polynomials perfectly approximated by quadrature rules in finite dimensional spaces.

\subsection{Backward scheme}
\label{subse de backward scheme}

As mentioned before, in the Markovian setting it is possible to express the solution to the BSDE equation \eqref{Eq:Bwd} in terms of the so-called \emph{decoupling field}, that is, applications $u: \R_+, \R^d \rightarrow \R$ and $v: \R_+, \R^d \rightarrow \R^d$ defined by \[u(t,x)= Y^{t,x}_t ;\qquad v(t,x)=Z^{t,x}_t.\]

For $\gamma \geq 1$, we consider a time grid of the form 
\begin{equation} 
t_i = T\mbra{1-\mpar{1-\frac{i}{n} }^\gamma}
\label{Eq:DefTk}
\end{equation}
and we set $h_{i}:= t_{i+1}-t_i$ for $i=0,\ldots, n$.  We study a cubature based  {Bouchard-Touzi-Zhang scheme} defined by
\begin{itemize}
\item[(i)] Terminal condition is $(\hu_n,\hv_n)=(g,0)$\\
\item[(ii)] Transition from step $i+1$ to step $i$ given by
\begin{align*}
\hu_i(x) &= \E^{\hQ} \mbra{\hu_{i+1}(\hat{X}_{t_{i+1}}^{t_i,x})  + h_i f(x,\hu_i(x),\hv_i(x))},
\\
\hv_i(x)&= \E^{\hQ}\mbra{\hu_{i+1}(\hat{X}_{t_{i+1}}^{t_i,x}) \frac{\Delta \homega_i}{h_i}},
\end{align*}
\end{itemize}
where $\Delta \homega_i= \homega_{t_{i+1}}-\homega_{t_i}$, and $\E^{\hQ}$ is the expectation with respect to the cubature measure. We then define $\hat{Y}_i=\hu_i(\hX_{t_i})$ and $\hZ_i=\hv_i(\hX_{t_i})$.

\subsection{Main results}

\subsubsection{Error expansion}

Our first result, Theorem \ref{Thm:ExpansionFwdCumulated} extends the results of \cite{talay_expansion_1990} on the Euler scheme for  weak approximations of SDEs to the case where the underlying numerical method is not the Monte Carlo method but the cubature on Wiener spaces.

\begin{Theorem}[Forward error expansion]\ 
Set $m\geq 3$. Suppose that Assumption \ref{H0} (resp. \ref{H1}) holds with $M\geq m+2$,  and take $\hQ$ to be a cubature measure from a cubature formula of order $m$ on a uniform (resp. decreasing) step grid with $\gamma=1$ (resp. $\gamma > m-1$).

Then, there is a constant $K$ such that, for all $i<n$,
\begin{equation}
|\E^{\hQ} \mbra{g (X^{t_i,x}_T)} - \E \mbra{g (X^{t_i,x}_T)} | \leq  K n^{ - \frac{m-1}{2}} \;.
\label{Eq:UniformCubBoundUniform} 
\end{equation} 

Moreover, if $M\geq m+3$ in  Assumption \ref{H0} (resp. \ref{H1})  {and the cubature is symmetric},  then
\begin{equation}
 |\E^{\hQ} \mbra{g (X^{t_i,x}_T)} - \E \mbra{g (X^{t_i,x}_T)}   - n^{ - \frac{m-1}{2}} \Psi^{lin}_{T}(t_i,x) | \leq K' n^{ - \frac{(m+1) \wedge \gamma }{2 } }  ,
 \label{Eq:CubExpansionFwd} 
\end{equation}
where the coefficient $\Psi^{lin}_{T}$ is given in Definition \ref{de psi-lin} below.

\label{Thm:ExpansionFwdCumulated}

\end{Theorem}

Our main result, see Theorem \ref{Thm:MainExpansionBwd} below, carries out  a similar analysis taking into account the non-linearity associated to the generator function $f$ in the formulation of BSDEs. 
Let us stress that this analogous result is obtained on a \emph{completely implementable} scheme to solve BSDEs, that is, we include the analysis on the conditional expectation approximation.

\begin{Theorem}[Backward error expansion]\  \label{Thm:MainExpansionBwd}
Suppose that Assumption \ref{H0} (resp. \ref{H1}) holds with $M\geq 9$,  and take $\hQ$ to be a cubature measure from a symmetric cubature formula of order $m\geq 3$ on a uniform (resp. decreasing) step grid with $\gamma=1$ (resp. $\gamma > m-1$). Then, for all $i<n$,
\begin{align}
 \left|  \hat{u}_{t_i}(x) - u(t_i,x) +n^{-1} \Psi^{nl}_{T}(t_i,x) \right|  \leq K n^{-2}\,,
\end{align}
where the coefficient $\Psi^{nl}_{T}$ is given in Definition \ref{de psi-nonlin} below.
 \end{Theorem}

 { \begin{Remark}
In both Theorems \ref{Thm:ExpansionFwdCumulated} and  \ref{Thm:MainExpansionBwd}, the role of the parameter $\gamma$ under Assumption \ref{H1}) is to allow for the needed regularisation of the boundary condition to take place. If $\gamma$ is taken to be smaller than the given bounds, we expect from the analysis to observe a suboptimal rate of convergence. In practice, for functions that are absolutely continuous this is not observed.
\end{Remark}
}

 { \begin{Remark}
It can be easily concluded from our development below that additional expansion terms in Theorems  \ref{Thm:ExpansionFwdCumulated} and  \ref{Thm:MainExpansionBwd} can be explicitly shown,  under stronger requirements on the parameters cubature parameter $m$, the regularity parameter $M$. Essentially, any additional term demands an increase of two on these parameters.
\end{Remark}
}
\subsubsection{Complexity reduction} 
\label{subsubse overview complexity}
With the above error expansion  at hand, one can implement a Romberg-Richardson method to increase the precision of the approximation profiting from the regularity of the value function $u$.
This is done by using an order $1$ method, here a modified Euler scheme, which is often easier to implement than a high order scheme and may exhibit  better numerical stability properties (see \cite{chassagneux_numerical_2015}  for a study of numerical stability in the case of BSDEs approximation).
With an increased precision, one can hope to lower the numerical complexity of the method. However, this is not the case here if one simply uses the cubature tree to compute the conditional expectation. 

Indeed, the main drawback of the cubature method is its complexity growth with respect to the number of time discretization steps on the approximation. 
In the case of approximation of expectations, complexity can be controlled using  reduction techniques as high order recombination \cite{litterer_high_2012} and TBBA (Tree Based Branching Algorithm) \cite{crisan_minimal_2002}. 

To make the extrapolation method worth implementing in practice, we introduce an extra projection on a finite grid. In Section \ref{se complexity}, we show, in an abstract setting of interpolation operators, how the backward scheme should be modified. 
We also exhibit sufficient conditions that the interpolation operators should satisfy to extend the convergence results to the modified scheme. 
We illustrate this findings by considering the example of multi-linear interpolation. 
Theorem \ref{th linear projection} proves the gain in complexity by using this approach.

\vspace{5pt}
The rest of the paper is organised as follows. In Section \ref{se forward}, we study the approximation of the BSDE in the linear case, associated to the cubature method. In Section \ref{Sec:BwdAnalysis}, we obtain an error expansion for the general case $f\neq 0$. In Section \ref{se complexity}, we focus on the complexity reduction via the use of interpolation operators and extrapolation methods and give a numerical illustration of our result. Finally, the Appendix collects some useful results  on integral approximation by Riemann sums and reviews the numerical implementation.

\section{Convergence analysis for the forward process}
\label{se forward}

We first provide an error extrapolation for the BSDEs when $f=0$. 
This result is new in the context of cubature method and terminal condition with Lipschitz only regularity. 

\subsection{Notation}

In this section, we highlight the notation we use in all the document and that might not be considered completely standard.

\subsubsection{Multi-indices}
\label{subsec:multiindex}
Multi-indices allow to easily manage differentiation and integration in several dimensions. 
Let us consider the set of multi-indices
\begin{equation}\label{Eq:defM} 
\mc{M} =  \{ \emptyset \} \cup \bigcup_{l\in\NN^*} \ \{0,1,\ldots,d\}^l,
\end{equation}
where $\emptyset$ refers to the zero-length multi-index.  
We define ``$*$'' to be the natural concatenation operator, and we consider some norms in $\mc{M}$. For $\beta=(\beta_1,\ldots,\beta_l)$:
\[ |\beta|_{p}  = \sum_{i=1}^l |\beta_i|^p , \text{ for } p\in \NN ; \qquad |\beta|_{\infty} = \max_{i} |\beta_i|,\]
and
 \[  \|\beta\| := (\#\{i:\beta_i =0\}) + |\beta|_0.\] 
 
 Naturally $|\emptyset|_p= \|\emptyset\|=0$. 
 For every $\beta\neq \emptyset$, we set $\beta^{>0}$ the multi-index obtained by deleting the zero components of $\beta$, and
 \begin{align}
\label{Def:beta_ops}  
\beta_{i:j}&:=  (\beta_i,\ldots,\beta_j)  \  \text{ for } 1\leq i<j\leq|\beta|_0; \\
 -\beta &:= (\beta_2,\ldots,\beta_l) \qquad  \beta-  := (\beta_1,\ldots,\beta_{l-1}). \notag
 \end{align}

We refer to the set of multi-indices of degree at most $l$ denoted by $\mc{A}_l:=\{\beta \in \mc{M}: \| \beta\|\leq l\}$, and to its \emph{frontier set}  \( \partial\mc{A}:=\{\beta \in \mc{M}\setminus\mc{A}: -\beta\in \mc{A}\}.\) 
It is readily seen that $\partial\mc{A}_l \subset \mc{A}_{l+2}\setminus \mc{A}_l$.\

\subsubsection{Stratonovich differential operators }

As we show in section \ref{Sec:Cubature} (see also  \cite{lyons_cubature_2004}),  the cubature on Wiener spaces method uses the algebraic structure of the differential operators associated to the Stratonovich integral.  
We introduce then a special notation for the Stratonovich operators associated to the SDE \eqref{Eq:Fwd} and their iterated actions. 

Let $\psi:\R_+\times \R^d \rightarrow \R$, then we define 
\begin{align*} 
V^\emptyset \psi(t,x) = & \psi(t,x) \\
V^{(0)} \psi (t,x) = & \sum_{i=1}^d \bb_i \partial_{x_i} \psi (t,x) + \partial_t \psi(t,x)  \\
 V^{(j)} \psi (t,x) = & \sum_{i=1}^d \sigma_{i,j} \partial_{x_i} \psi (t,x)   ;  j=1, \ldots, d 
 \end{align*}
 and for any multi-index $\beta \neq \emptyset$ with $|\beta|>1$, 
\[ V^{\beta}(t,x) = V^{(\beta_1)} [V^{-\beta} \psi](t,x) \]
with $-\beta$ as in \eqref{Def:beta_ops}.

\subsubsection{Space of differentiable functions}
We denote by $\bar{C}^m_b$ the class of differentiable functions $\psi:\R^d\rightarrow \R$ with bounded derivatives $V^\beta \psi $ for every $\beta \in \mc{A}_m$, i.e. functions whose  semi-norm defined by
\begin{equation}
|| \psi ||_{m,\infty} : =  \sup_{\beta \in \mc{A}_m} | V^\beta \psi |_{\infty}.
\label{Def:Seminorm}
\end{equation}
is finite. In the definition, $|\cdot|_{\infty}$ stands for the usual maximum norm. 

For convenience, for a fixed discretization grid, and $\psi:[0,T]\times\R^d\rightarrow \R$, we write 
\begin{equation}
|| \psi ||_{i;m,\infty} : =  \sup_{t\in [t_i,t_{i+1}]}\| \psi(t,.) \|_{m,\infty} .
\label{Def:Seminorm2}
\end{equation}

\subsubsection{Ito differential operators}
Let $\psi:\R_+\times \R^d \rightarrow \R$, then we define 
\begin{align*} 
L^{(0)} \psi (t,x) = & \partial_t \psi(t,x) + \sum_{i=1}^d b_i \partial_{x_i} \psi (t,x)   + \frac 1 2\sum_{i=1}^d\sum_{j=1}^d \sum_{k=1}^d \sigma_{ik}\sigma_{jk} \partial_{x_i x_j} \psi (t,x) .
 \end{align*}
Note that this operator can be written in terms of the Stratonovich differential operators we introduced, as follows
\[  L^{(0)} =  V^{(0)} + \frac{1}{2}\sum_{j=1}^d V^{(j,j)}. \]

\subsubsection{Iterated integrals} 

Let $\omega$ and $\zeta$ be two functions in $C_{\mathrm{bv}}([0,T],\R^d)$, the space of continuous functions with bounded variation defined from $[0,T]$ to $\R^d$ .
Let us define the iterated integral of $\zeta$ with respect to $\omega$ by
\[I^\beta_{s,t} (\zeta, \omega) := \int_{s<t_1<\ldots<t_{|\beta|}\leq t}  \zeta_{t_1} \ud\omega^{\beta_1}({t_1})  \cdots  \ud\omega^{\beta_{|\beta|}}({t_{|\beta|})}    \]
where $\omega^{i}$ is the i-th component of $\omega$, and we fix by convention $\omega^0(t):=t$.   In the following, we write $I^\beta_{s,t} (\omega)=I^\beta_{s,t}(1,\omega)$.

We introduce a similar notation to represent iterated integrals with respect to the Brownian motion. 
Indeed, let $X$ be an adapted process. We set
\[J^\beta_{s,t} (X):= \int_{s<t_1<\ldots<t_{|\beta| < t}} X_{t_1}  \circ \ud W^{\beta_1}_{t_1}  \cdots \circ \ud W^{\beta_{|\beta|}}_{t_{|\beta|}}    \]
where the notation $\circ$ denotes the Stratonovich integral, and we keep the convention $W^0_t = t$. Moreover, set $J^\beta_{s,t} :=J^\beta_{s,t}(1)$.

\subsubsection{Stochastic flow}
In what follows, for any $0\leq s\leq t <T$, we denote  by $X_t^{s,x}$ the process $X_t$ conditioned to be equal to $x$ at time $s<t$, i.e.
\[X_t^{s,x}=  x + \int_s^t b(r,X^{s,x}_r) \ud r + \int_s^t \sigma(r,X_r^{s,x}) \ud W_r .\]

\subsubsection{Operators}
We consider now a family of operators associated with the SDE \eqref{Eq:Fwd} and the process \eqref{Def:XCub}.
\begin{Definition}[Operators]
We denote by $P$ the operator over measurable functions associated to the diffusion $X$ defined, for any measurable function $\psi: [0,T]\times \R^d \rightarrow \R$ , by
\[ P_{s,t}\psi(s,x):= \E[\psi(t,X^{s,x}_t)].\]
Similarly, we define by $Q$ the analogous cubature operator given by
\[ Q_{s,t} \psi(s,x) := \E^{\hQ}[ \psi (t, \hat{X}^{s,x}_{t})].\]

By a slight abuse of notation,  we use this notation also for one-parameter functions, that is, if $\psi:\R^d\rightarrow \R$, $P_{s,t} \psi$ (respectively $Q_{s,t} \psi$) denotes the operator applied to the function $\bar{\psi}:[0,T]\times \R^d \rightarrow \R$ such that $(t,x )\rightarrow \bar{\psi}(t,x):= \psi(x)$.

\label{Def:Operators}

\end{Definition}

Note that, by Definition \ref{Def:Operators} and \eqref{Def:XCubCond}, we have
\[ Q_{s,t}Q_{t,u} \psi(s,x) = \E^{\hQ} [ Q_{t,u}  \psi(t,\hX^{s,x}_t) ] = \E^{\hQ} [ \psi(u,\hX_{u}^{t,\hX^{s,x}_t}) ] =    \E^{\hQ} [ \psi(u,\hX_{u}^{s,x}) ] = Q_{s,u} \psi(s,x).\]

\subsection{Cubature on Wiener spaces}
\label{Sec:Cubature}

The cubature measure on Wiener spaces was introduced in \cite{lyons_cubature_2004}, as a tool to construct weak approximations of functionals of the Brownian motion. 
It generalises the quadrature method to an infinite dimensional space: It aims at approximating the  Wiener measure restricted to a time interval $[0,T]$ with a finite probability defined on $C([0,T], \R^d)$. 
As in the quadrature method, this approximation consists in preserving the exact value of the expectation of some basic functionals that will play a similar role as the one played by polynomials in finite dimensions.

\begin{Definition}[Cubature formula \cite{lyons_cubature_2004}]
\label{Def:CubForm}
Let $m$ be a natural number. 
An $m$-cubature formula on the Wiener space $C^{0}([0,1],\R^d)$ is a probability measure ${\Q}$ with finite support on $C_{\mathrm{bv}}^{0}([0,1],\R^d)$ (continuous functions and bounded variation starting in 0) such that the expectation of the iterated Stratonovich integrals of degree $m$ under the Wiener measure and under the cubature measure ${\Q}$ are the same, i.e., for all multi-index $\beta \in \mc{A}_m$
\[
\E \mbra{J^{\beta}_{0,1}}  = \E^{{\Q}} \mbra{ I^{\beta}_{0,1}(\omega) }:= \sum_{j=1}^{\kappa}\theta_j \mbra{I^\beta_{0,1}(\omega_j)}
\]
where $\{\omega^1,\ldots,\omega^\kappa\}$ and $\theta_1,\ldots,\theta_\kappa$ are respectively the support and weights of the finite measure ${\Q}$.
\end{Definition}

Examples of cubature formulas of order 3 and 5 are known for arbitrary dimensions. 
Higher order cubature methods (up to order 11) are also given for small dimensions, see e.g. \cite{lyons_cubature_2004,gyurko_efficient_2011}.

 Definition \ref{Def:CubForm} may be extended to an $m$-cubature formula on the Wiener space $C^{0}([0,t],\R^d)$, for an arbitrary $t>0$. 
Indeed,  the rescaling properties of the Brownian motion imply that $\omega^1,\ldots,\omega^\kappa$ and $\theta_1,\ldots,\theta_\kappa$  form an $m$-cubature formula on  $C^{0}_{\mathrm{bv}}([0,1],\R^d)$ if and only if $t^{1/2}\omega^1,\ldots,$ $t^{1/2}\omega_\kappa$ and $\theta_1,\ldots,\theta_\kappa$ form an $m$-cubature formula on  $C^{0}_{\mathrm{bv}}([0,t],\R^d)$.
This justifies the choice of giving the definition on the interval $[0,1]$.

\begin{Definition}[Symmetric cubature formula] \
\label{Def:SymCub}
We say that a cubature formula is symmetric if for any path $\omega^* \in \supp(\Q)$ then  $-\omega^* \in \supp(\Q)$ and $\Q[\omega = \omega^*] = \Q[\omega = -\omega^*] $. 
\end{Definition}

\begin{Remark}
The  properties of Brownian Motion imply that if $||\beta||=2k+1$ for some integer $k$, then $ \E \mbra{J^{\beta}_{0,1}}  = 0$.
Cubature formulas with the  symmetry properties will therefore approximate exactly iterated integrals of any odd degree.  {Thus, without loss of generality, we assume that any symmetric cubature measure is of odd order.}
\label{Rmk:Symmetry}
\end{Remark}

Lyons and Victoir \cite{lyons_cubature_2004} showed that the cubature method in the space  $C^{0}([0,t],\R^d)$ provides a weak approximation to the Brownian motion with an error bounded by some power of the time length of approximation $t$.  
The construction can then be iteratively applied on small intervals to obtain an approximation with a good control for a time interval $[0,T]$ with arbitrary length. This motivates the following definition.

\begin{Definition}[Cubature measure] \ 
\label{Def:Cubature}
Let $0=t_0 < \ldots< t_n = T$, and let $h_i = t_{i+1}-t_i$. Given a cubature formula $\tQ$ represented by a set of weights $\theta_1, \ldots, \theta_\kappa$ associated to a set of paths $\omega^1, \ldots, \omega^\kappa$ in $C^{0}_{\mathrm{bv}}([0,1],\R^d)$, we build the probability space  $(\hOmega,\hQ)$, where $\hOmega =  C^{0}_{\mathrm{bv}}([0,T],\R^d)$  and $\hQ$ is a finite measure with support on paths indexed by $\eta \in \mbrace{1,\ldots,\kappa}^n$ given by
\[ \hat{\omega}^\eta (t) = \sum_{i=1}^n  h_i^{1/2} \omega^{\eta_i}  \mpar{ \frac{(t_{i+1} \vee (t \wedge t_{i}) )-t_i}{h_i}} \one{t> t_{i-1}}  \]
with associated probability $\hat{\theta}_\eta = \theta_{\eta_1}\theta_{\eta_2}\cdots \theta_{\eta_n}$.
\end{Definition}

Let us emphasize that the cubature measure constructed according to Definition \ref{Def:Cubature} depends on the cubature formula $\tilde{\Q}$ and on the grid $\T = \{t_0, \ldots, t_n\}$.

\subsection{Forward Error expansion}

In this section we show an expansion for the approximation error of conditional expectations when using the cubature method coupled with an Euler scheme. 
The idea of the proof follows the well-known approach of combining a one-step expansion with a global stability property. 
Importantly, the analysis relies on a good regularity property of the function  being approximated, as required by the cubature operator.

In all our development $K,K',\ldots$ denote constants  that might depend on the parameters of the problem (i.e. $T,f,g,b,\sigma,x_0)$ or on cubature measure that we assume fixed,  but not on the parameters of the scheme. 
We use the convention that their value might change from line to line.

We start by recalling regularity properties of conditional expectation under our standing assumptions.
\begin{Proposition}[Regularity]  
Under Assumption \ref{H0}, there exists a constant $K$ such that for all for all $(t, x) \in [0,T]\times R^d$ and $0<k\leq M$, 
\begin{equation} 
\left\|\E\mbra{g(X_T^{t,x})} \right\|_{k,\infty} \leq K |g(x)|_{k,\infty}.
\label{Eq:lin_u_SB}
\end{equation}
Under  Assumption \ref{H1}, there exists a constant $K$ such that for all for all $(t, x) \in [0,T]\times R^d$ and $0<k\leq M$,
\begin{equation} 
\left\| \E\mbra{g(X_T^{t,x})} \right\|_{k,\infty} \leq K |g|_{\mathrm{Lip}}(T-t)^{ - \frac{k-1}{2}}.
\label{Eq:lin_u_UFG_LB}
\end{equation}
If  in addition $g$ has bounded derivatives up to order $p$, then
\begin{equation} \left\| \E\mbra{g(X_T^{t,x})} \right\|_{k,\infty} \leq K |\nabla^p g|_{\infty}(T-t)^{ - \frac{k-p}{2}}.
\label{Eq:lin_u_UFG_SB}
\end{equation}
\label{Prop:lin_u_reg}
\end{Proposition}

Equation \eqref{Eq:lin_u_SB} is a consequence of the iterated application of Ito's formula, under Assumption \ref{H0}.  
The claims under the UFG condition are proved in  \cite{crisan_cubature_2013}: \eqref{Eq:lin_u_UFG_LB} is proved in Corollary 78 and \eqref{Eq:lin_u_UFG_SB}  is deduced from extending the arguments of Corollary 32 with   Corollary 78. The reader may refer to the PhD thesis \cite{nee_sharp_2011} for further results on gradient bounds under alternative conditions.

\subsubsection{One-step expansion}

The following approximation result is a restatement of  the results  in \cite{lyons_cubature_2004} (see also  Section 3.4 in \cite{crisan_cubature_2013} ) .

\begin{Proposition}
Under Assumption \ref{H0} (resp. \ref{H1}), let $\psi: [0,T]\times \R^d \rightarrow \R$ be a bounded function in $\bar{C}^{m+2}_b$ uniformly in $t$.
Then, for any $i\leq n-1$ (resp. $i\leq n-2$),

\begin{align} 
|Q_{t_i,t_{i+1}}& \psi (t_i,x) - P_{t_i,t_{i+1}}\psi(t_i,x)  |  \label{Eq:OneStepCubBound} \\
& \leq K_{m+1} h_{i}^{ \frac{m+1}{2}} ||\psi ||_{i;m+1,\infty} + K_{m+2} h_{i}^{ \frac{m+2}{2}} ||\psi ||_{i;m+2,\infty},  \notag
\end{align}
where $||.||_{i,.,\infty}$ is defined in \eqref{Def:Seminorm2} and   $K_j = \sum_{\beta\in \mc{A}_j \setminus \mc{A}_{j-1} }|C_\beta|$, with $C_\beta = \E^{\hQ} I^\beta_{0,1} - \E J^\beta_{0, 1}$. 

If in addition $\psi \in \bar{C}_b^{m+3}$, 
\begin{align}
& \left |Q_{t_i,t_{i+1}}\psi(t_i,x) - P_{t_i,t_{i+1}}\psi(t_i,x)  - h_{i}^{ \frac{m+1}{2}}\sum_{\beta \in \mathcal{A}_{m+1} \setminus \mc{A}_m } C_\beta V^{\beta} \psi(t_i,x)    \right |    \label{Eq:OneStepCubExactNoSym}\\
& \quad\qquad \leq K_{m+2} h_{i}^{ \frac{m+2}{2}}  ||\psi ||_{i;m+2,\infty}  + K_{m+3} h_{i}^{ \frac{m+3}{2}}  ||\psi||_{i;m+3,\infty}. \notag
\end{align}

Moreover, if  $m$ is odd,  the cubature measure $\Q$ is symmetric and $\psi \in \bar{C}_b^{m+4}$, then 
\begin{align}
& \left |Q_{t_i,t_{i+1}}\psi(t_i,x) - P_{t_i,t_{i+1}}\psi(t_i,x)  - h_{i}^{ \frac{m+1}{2}}\sum_{\beta \in \mathcal{A}_{m+1} \setminus \mc{A}_m } C_\beta V^{\beta} \psi(t_i,x)    \right |    \label{Eq:OneStepCubExact}\\
& \quad\qquad \leq K_{m+3}   h_{i}^{ \frac{m+3}{2}}  ||\psi ||_{i;m+3,\infty} + K_{m+4} h_{i}^{ \frac{m+4}{2}}  ||\psi ||_{i;m+4,\infty} . \notag
\end{align}
\label{Prop:OneStepCub}
\end{Proposition}

 {Let us remark that the constants $C_\beta$ (and thus also the constants $K_{m+1},K_{m+2}$)  depend only on the chosen cubature method and are explicit. For instance, the canonical cubature of order 3 in dimension 1 (see Section \ref{Sec:NumericalTest} below )  gives $K_4 = 2, K_5=0$. }\\

\proof
From the Taylor-Stratonovich expansion   (see Theorem 5.6.1 in  \cite{kloeden_numerical_1992})  applied to $\psi$, one gets, for any $x\in \R^d$,
\[  P_{t_i,t_{i+1}}\psi(t_i,x) =  \E \mbra{\psi(t_{i+1},X_{t_{i+1}}^{t_i,x})} = \sum_{\beta \in \mathcal{A}_m} V^{\beta} \psi(t_i,x) \E J^\beta_{t_i, t_{i+1}} + \sum_{\beta \in \partial \mathcal{A}_m} \E J^\beta_{t_i,t_{i+1}}(   V^{\beta} \psi(.,X^{t_i,x}_.)  ).\]
Similarly,  a Taylor expansion and the definition of $\hX$ in \eqref{Def:XCub} shows
\[ Q_{t_i,t_{i+1}}\psi(t_i,x) = \E^{\hQ} \mbra{\psi(t_{i+1} ,\hX_{t_{i+1}}^{t_i,x})} = \sum_{\beta \in \mathcal{A}_m} V^{\beta} \psi(t_i,x) \E^{\hQ} I^\beta_{t_i,t_{i+1}} + \sum_{\beta \in \partial \mathcal{A}_m} \E^{\hQ} I^\beta_{t_i,t_{i+1}}(   V^{\beta} \psi(.,\hX^{t_i,x}_.) ).\]
The construction of the cubature measure in Definition \ref{Def:Cubature}  and the definition of the cubature formula, imply that $ \E^{\hQ} I^\beta_{t_i,t_{i+1}}  = \E J^\beta_{t_i, t_{i+1}} $ for all $\beta \in \mc{A}_m$. 
Hence,  by using the regularity assumptions on $\psi$, the scaling properties of the cubature and the Markov property, one gets 
\begin{align}  
[Q_{t_i,t_{i+1}} & - P_{t_i,t_{i+1}}]\psi(t_i,x)  \label{Eq:CubError} \\
=  & \sum_{\beta \in \partial \mathcal{A}_m} \E^{\hQ} I^\beta_{t_i,t_{i+1}}(   V^{\beta} \psi(.,\hX_.) ) -  \E J^\beta_{t_i,t_{i+1}}(   V^{\beta} \psi(.,X^{t_i,x}_.)  ) \notag\\ 
\leq &   \sum_{\beta \in \partial \mathcal{A}_m}  h_{i}^{ \frac{||\beta||}{2}}  \sup_{t\in[t_i,t_{i+1})} |V^{\beta} \psi (t,.)|_{\infty}  \mpar{ \E^{\hQ} I^\beta_{0,1} - \E J^\beta_{0,1}}
\notag\\
\leq &   K_{m+1} h_{i}^{ \frac{m+1}{2}} ||\psi||_{i;m+1,\infty}   +   K_{m+2} h_{i}^{ \frac{m+2}{2}} ||\psi ||_{i;m+2,\infty}
\notag
\end{align}
where for  $j=m+1,m+2$, \[K_j =  \sum_{\beta \in  \mathcal{A}_{j}\setminus \mc{A}_{j-1} }  \left| \E^{\hQ} I^\beta_{0,1} - \E J^\beta_{0,1} \right|.\]

Having shown \eqref{Eq:OneStepCubBound}, let us now assume $\psi \in \bar{C}_{b}^{m+3}$. 
Then we can expand the terms in $\mc{A}_{m+1}$  in the equality in \eqref{Eq:CubError}, and apply the scaling properties of the cubature to get
\begin{align*}
 \bigg| [Q_{t_i,t_{i+1}} & - P_{t_i,t_{i+1}}]  \psi(t_i,x)   - \sum_{\beta \in \mathcal{A}_{m+1} \setminus \mathcal{A}_m} h_{i}^{ \frac{||\beta||}{2}} V^{\beta} \psi(t_i,x) \mbrace{\E^{\hQ} I^\beta_{1} - \E J^\beta_{1}} \bigg| \\
 =  & \sum_{\beta \in \partial \mathcal{A}_{m+1}} \E^{\hQ} I^\beta_{t_i,t_{i+1}}(   V^{\beta} \psi(.,\hX_.) ) -  \E J^\beta_{t_i,t_{i+1}}(   V^{\beta} \psi(.,X^{t_i,x}_.)  ) \\ 
\leq & K_{m+2} h_{i}^{ \frac{m+2}{2}}   ||\psi||_{i,m+2,\infty} + K_{m+3} h_{i}^{ \frac{m+3}{2}} ||\psi ||_{i,m+3,\infty}.
\end{align*}

Finally, if $\psi$ is $m+4$ times differentiable, then we can expand once more and use the symmetry of the cubature formula and Remark \ref{Rmk:Symmetry} to obtain the result.

\eproof

\subsubsection{Global expansion: Proof of Theorem \ref{Thm:ExpansionFwdCumulated}}

Using the previous one-step results, we can study an explicit error expansion for the error in the cubature approximation of $X$ for several steps, thus extending \cite{talay_expansion_1990} to the cubature approximation. 
We analyze this property under both Assumptions \ref{H0} and \ref{H1}. We start with the following control result.

\begin{Lemma}
Suppose that either  Assumption \ref{H0}  or Assumption \ref{H1} hold with $M \geq m+2$. Assume that $(t_0, \ldots, t_n)$ are defined as in \eqref{Eq:DefTk}, with $\gamma=1$ if Assumption \ref{H0} holds and $\gamma \geq M-2$ otherwise.  Then, there is a constant $K$ such that, for all $0< k\leq M$,
\[\sum_{i=0}^{n-2}  h_i^{\frac{k}{2}}  ||  P_{.,t_n} g  ||_{i;k,\infty}  \leq K n^{ - \frac{k}{2} +1} .\] 
\label{Lem:ControlSumLinear}
\end{Lemma}

\proof
When Assumption \ref{H0} holds, we have that $h_i= Tn^{-1}$ and from Proposition \ref{Prop:lin_u_reg},
\[\sum_{i=0}^{n-2}  h_i^{\frac{k}{2}}  ||  P_{.,t_n} g  ||_{i;k,\infty} \leq \sum_{i=0}^{n-2} K  n^{-\frac{k}{2}}  ||g (.) ||_{k,\infty}  \leq K' n^{ - \frac{k}{2} +1}.\] 

On the other hand, if  Assumption \ref{H1} holds, we have from  Proposition \ref{Prop:lin_u_reg} 
 {that $||  P_{.,t_n} g  ||_{i; k,\infty} \leq K |g|_{Lip} (T-t)^{-\beta}$ for $\beta = \frac{k-1}{2}$. Then, we get from Corollary \ref{Cor:ControlSum2},
\[\sum_{i=0}^{n-2}  h_i^{\frac{k}{2}}  ||  P_{.,t_n} g  ||_{i; k,\infty} \leq     K' n^{- \frac{k}{2}+1}.\] 
since in this case $\gamma (\frac{k}{2} - \beta ) = \frac \gamma 2 \geq \frac {M-2} 2 \geq \frac{k}{2}-1 $.}
\eproof

Before proving the main expansion result for this part i.e. Theorem \ref{Thm:ExpansionFwdCumulated}, let us make precise the shape of the leading coefficient $\Psi^{lin}$ in the expansion.
\begin{Definition}{\ }\label{de psi-lin}
\begin{enumerate}[i.]
\item Under Assumption \ref{H0},
\begin{equation} 
\Psi^{lin}_{T}(s,x) :=  T^{ \frac{m-1}{2}}\E\mbra{ \sum_{\beta\in \mc{A}_{m+1}\setminus \mc{A}_m} \int_{s}^{T} {C}_{\beta} V^\beta [P_{t,T}g(X_{t}^{0,x})] \ud t }.
\label{Eq:UniformCubExpansionUniform}
\end{equation}
\item Under Assumption \ref{H1}
\begin{align*} 
\Psi^{lin}_{T}(s,x) := (T^{\frac{1}{\gamma}}\gamma)^{\frac{m-1}{2}}\E\mbra{ \sum_{\beta\in \mc{A}_{m+1}\setminus \mc{A}_m} \int_{s}^{T} {C}_{\beta} V^\beta [P_{t,T}g(X_{t}^{0,x})] \mpar{T-t }^{(1-\frac{1}{\gamma}) \frac{m-1}{2}} \ud t }.
\label{Eq:UniformCubExpansionH1}
\end{align*}
\end{enumerate}
\end{Definition}

We are ready to prove the global error expansion result for the conditional expectation using the cubature method.\\

\noindent \textbf{Proof of Theorem \ref{Thm:ExpansionFwdCumulated}}

Let $u$ be given as the solution of the linear equation $\partial_t u(t,x) + \mc{L}u(t,x) =0 $ defined on $[0,T)$ with boundary condition $u(T,x)=g(x)$ and $\mc L$ the Dynkin operator of $X$. Clearly, for any $i=0, \ldots, n-1$, we have $  P_{t_i, t_{i+1}} u(t_i,x) = u(t_i,x).$ In particular, \[u(0,x) = P_{0,T}  u(0,x)  = \E[ u(T, X^{0,x}_{T}) ] = \E[ g(X^{0,x}_{T}) ]  = P_{0, T } g(x)\]

Note that given that $M\geq m+2$, Assumptions  \ref{H0} (resp.  \ref{H1}) and Proposition \ref{Prop:lin_u_reg} imply that $u(t_i,.)$ has bounded derivatives of all order up to $m+2$, for $i<n$. 

Using the definition of the function $u$ and the properties of the family of operators $Q$ and $P$, we can then reformulate the error term as a telescopic sum, 
\begin{align} 
\E^{\hQ}[ g(\hX^{0,x}_{T}) ] - & \E[ g(X^{0,x}_{T}) ]  =  Q_{0,T}u(0,x) - P_{0,T}u(0,x) \notag\\
= & \sum_{i=0}^{n-1} Q_{0,t_i} \{[Q_{t_i,t_{i+1}}-P_{t_i,t_{i+1}}] P_{t_{i+1},T} u\}(0,x) \notag\\
= & \sum_{i=0}^{n-2} Q_{0,t_i}\{[Q_{t_i,t_{i+1}}-P_{t_i,t_{i+1}}]  u\}(0,x)  + Q_{0,t_{n-1}}\{ [Q_{t_{n-1},t_n} -P_{t_{n-1},t_n}]g\}(x) .   \label{Eq:CubExp}
\end{align}

At this point, we divide proof of the theorem in two parts. \\

\textit{Proof of the bound \eqref{Eq:UniformCubBoundUniform} }\\

We consider a bound for each of the two terms in \ref{Eq:CubExp}.  To treat the first term,  we use \eqref{Eq:OneStepCubBound} and  Lemma \ref{Lem:ControlSumLinear} to deduce
\begin{align}
 \sum_{i=0}^{n-2} Q_{0,t_i}\{[Q_{t_i,t_{i+1}}-P_{t_i,t_{i+1}}]  u\}(0,x)     \leq & K_{m+1} \sum_{i=0}^{n-1}  h_i^{ \frac{m+1}{2}} ||u||_{i;m+1,\infty}  +  K_{m+2} \sum_{i=0}^{n-1}  h_i^{ \frac{m+2}{2}} ||u||_{i;m+2,\infty}  ]\label{Eq:CubBound0} \\ \leq &  K' n^{ - \frac{m-1}{2}} . \notag
\end{align}
Now, the last term in  \ref{Eq:CubExp} is simply an average of a one step cubature approximation. 
Under Assumption \eqref{H0}, $g$ is regular enough to use the one-step cubature approximation result in Proposition  \eqref{Prop:OneStepCub}, hence
\begin{equation}  
|[Q_{t_{n-1},t_n} -P_{t_{n-1},t_n}]g(x)|_{\infty} \leq K h_{n-1}^{\frac{m+1}2} ||g||_{m+2,\infty} = K n^{-\frac{m+1}2} ||g||_{m+2,\infty}  .
\label{Eq:LastTermFwdH0}
\end{equation}
The same result cannot be applied under Assumption \eqref{H1} though, as $g$ is not supposed to be regular. However using the Lipschitz regularity of $g$, the bounded variation of the cubature formula and its rescaling, we get
\begin{equation}
 \begin{split} |[Q_{t_{n-1},t_n} -P_{t_{n-1},t_n}]g(x)|_{\infty}   \leq & |Q_{t_{n-1},t_n}[g](x) -g(x)|_{\infty} + | P_{t_{n-1},t_n}[g](x) -g(x)|_{\infty}\\ \leq & K |g|_{Lip} h_{n-1}^{\frac 12} = K |g|_{Lip} n^{- \frac \gamma 2}     \end{split}
 \label{Eq:LastTermFwdH1}
\end{equation} 

The first claim is thus proven.\\

\textit{Explicit first expansion term \eqref{Eq:CubExpansionFwd}}\\

 { In view of \eqref{Eq:LastTermFwdH0}  (or \eqref{Eq:LastTermFwdH1} under Assumption \ref{H1}), we just need to obtain an explicit approximation of the first term in \eqref{Eq:CubExp}. We proceed in three steps: first, we identify it as the sum of the residuals of one step approximations in \eqref{Eq:CubBound1}.  This term includes the approximation operator $Q$, which would depend on the scheme. Then,  in a second step,  we show  that we can replace this operator by the operator $P$ to the cost of a higher order term in \eqref{Eq:ControlLinearA}. Finally, we show that we can replace the sum by an integral also to the cost of a higher order term in  \eqref{Eq:CubBound3}}.

\begin{enumerate}[i.]
\item We start with the identification of the first order in the expansion.   We use the extra regularity assumptions  to apply successively, Jensen's inequality, properties of expectation operators, the claim on symmetric cubatures of Proposition \ref{Prop:OneStepCub} and Lemma \ref{Lem:ControlSumLinear}, to  write
\begin{align}
\bigg|  \sum_{i=0}^{n-2} & Q_{0,t_i}\{[Q_{t_i,t_{i+1}}-P_{t_i,t_{i+1}}]  u\}(0,x)  -  \sum_{i=0}^{n-2}  \sum_{\beta \in \mc{A}_{m+1}\setminus \mc{A}_m}C_\beta  h_i^{ \frac{m+1}{2}}  Q_{0,t_i}  [V^\beta u ](0,x)  \bigg| \label{Eq:CubBound1}\\
& \leq  \sum_{i=0}^{n-2}   Q_{0,t_i}   \left|  \{ Q_{t_i,t_{i+1}}-P_{t_i,t_{i+1}} \}  u   -  \sum_{\beta \in \mc{A}_{m+1}\setminus \mc{A}_m}C_\beta  h_i^{ \frac{m+1}{2}}  [V^\beta u ]  \right| (0,x)  \notag  \\
& \leq  \sum_{i=0}^{n-2}  \left|    \{ Q_{t_i,t_{i+1}}-P_{t_i,t_{i+1}} \}  u   -  \sum_{\beta \in \mc{A}_{m+1}\setminus \mc{A}_m}C_\beta  h_i^{ \frac{m+1}{2}}  [V^\beta u ]   (0,x)\right|  \notag  \\
& \leq  K_{m+3} \sum_{i=0}^{n-2}  h_i^{ \frac{m+3}{2}}  || u||_{i;m+3,\infty}   + K_{m+4} \sum_{i=0}^{n-2}  h_i^{ \frac{m+4}{2}}  || u||_{i;m+4,\infty}  \notag\\
& \leq K n^{ - \frac{m+1}{2}}.\notag
\end{align}

\item Now, we analyse changing $Q$ by $P$ by showing the bound
\begin{equation}
\bigg|  \sum_{i=0}^{n-2}  \sum_{\beta \in \mc{A}_{m+1}\setminus \mc{A}_m}C_\beta  h_i^{\frac{m+1}{2}}  \left[ Q_{0,t_i}  [V^\beta u ](0,x) - P_{0,t_i}  [V^\beta u ](0,x)\right]  \bigg|   \leq K n^{-(m-1)}.
\label{Eq:ControlLinearA}
\end{equation}

To prove this, note  that from the regularity of $V^\beta u$, the definition of the seminorm and the bounds in Proposition \ref{Prop:lin_u_reg} we have
\[ || V^\beta u(t,.)||_{m+1,\infty} = || V^\beta P_{t,T} g(.)||_{m+1,\infty} \leq K || P_{t,T} g(.)||_{m+1+||\beta||,\infty} = K || u(t,.)||_{m+1+||\beta||,\infty}. \]
Hence, using equation \eqref{Eq:CubBound0} and reordering, we have 
\begin{align}
\bigg|  \sum_{i=0}^{n-2}  & \sum_{\beta \in \mc{A}_{m+1}\setminus \mc{A}_m}C_\beta  h_i^{\frac{m+1}{2}}  \left[ Q_{0,t_i}  [V^\beta u ](0,x) - P_{0,t_i}  [V^\beta u ](0,x)\right]  \bigg| \label{Eq:CubBound2}\\
& \leq  K  \sum_{i=0}^{n-2}     h_i^{\frac{m+1}{2}}  \sum_{\beta \in \mc{A}_{m+1}\setminus \mc{A}_m} \left[    \sum_{j=0}^{i-1}  h_j^{\frac{m+1}{2}} || V^\beta u||_{j;m+1,\infty}  +   \sum_{j=0}^{i-1} h_j^{\frac{m+2}{2}} ||V^{\beta} u||_{j;m+2,\infty}     \right]  \notag\\
& \leq  K \sum_{i=0}^{n-2}     h_i^{\frac{m+1}{2}}  \left[   \sum_{j=0}^{i-1}  h_j^{\frac{m+1}{2}} || u||_{j;2m+2,\infty}  +   \sum_{j=0}^{i-1} h_j^{\frac{m+2}{2}} ||u||_{j;2m+3,\infty}     \right]  \notag \\
& \leq  K    \sum_{j=0}^{n-2}      \left( \sum_{i=j+1}^{n-2}     h_i^{\frac{m+1}{2}}  \right) h_j^{\frac{m+1}{2}}  || u||_{j,2m+2,\infty}   +   \sum_{j=0}^{n-2}  \left( \sum_{i=j+1}^{n-2}      h_i^{\frac{m+1}{2}}  \right) h_j^{\frac{m+2}{2}} ||u||_{j,2m+3,\infty} .\notag 
\end{align}
If Assumption \ref{H0} holds, we deduce \eqref{Eq:ControlLinearA} directly using the boundedness in Proposition \ref{Prop:lin_u_reg}. 
Under Assumption  \ref{H1}, we deduce from \eqref{Eq:ControlSumPsiIntegral} in Lemma \ref{Lem:ApproxSum} 
\begin{align}
\sum_{j=0}^{n-2}     & \left( \sum_{i=j+1}^{n-1}     h_i^{\frac{m+1}{2}}  \right) h_j^{\frac{m+1}{2}}  || u(t_j,.)||_{2m+2,\infty}\\
& \leq  K \left(\frac{T\gamma}{n}  \right)^{\frac{m-1}{2}} \sum_{j=0}^{n-2}      \left(  \int_{t_{j+1}}^T \mpar{1-\frac{t}{T}}^{\frac{m-1}{2}(1-\frac{1}{\gamma})}\ud t  + R(j) \right) h_j^{\frac{m+1}{2}}  (T-t_{j})^{- \frac{2m+1}{2} }    \notag \\
& \leq  K \left(\frac{T^{1/\gamma}\gamma}{n}  \right)^{\frac{m-1}{2}} \sum_{j=0}^{n-2}      \left(  \int_{t_{j+1}}^T (T-t)^{\frac{m-1}{2}(1-\frac{1}{\gamma})}\ud t  \right) h_j^{\frac{m+1}{2}}  (T-t_{j})^{- \frac{2m+1}{2} }   \notag \\
& \qquad + \sum_{i=j+1}^{n-1} \frac{1}{n^{\frac{m-1}2}} \mpar{K' \frac1{n^{\gamma(\ell-1)+1}} + K'' \frac{1}{n^\ell} \mpar{T-t_{j+1}}^{\gamma(\ell-1)-\ell} }h_j^{\frac{m+1}{2}}  (T-t_{j})^{- \frac{2m+1}{2} }    \notag \\  
& \leq  K \left(\frac{T^{1/\gamma}\gamma}{n}  \right)^{\frac{m-1}{2}} \sum_{j=0}^{n-2}    (T-t_{j+1})^{ -\frac{m-1}{2}(1-\frac{1}{\gamma})+1 }    h_j^{\frac{m+1}{2}}  (T-t_{j})^{- \frac{2m+1}{2} } +O(n^{\frac{m(\ell+1)}{2}-1})  \notag\\
& \leq K n^{-(m-1)},\notag
\end{align}
where the last two inequalities follow from (the proof of) Corollary \ref{Cor:ControlSum2}.
A similar development can be applied to the second term in the last line of \eqref{Eq:CubBound2}, to deduce \eqref{Eq:ControlLinearA} in this case.

\item Finally, it remains to show that 
\begin{align}
\bigg| n^{ - \frac{m-1}{2}} \Psi^{lin}_{T}(0,x) -  \sum_{i=0}^{n-2}  \sum_{\beta \in \mc{A}_{m+1}\setminus \mc{A}_m}C_\beta  h_i^{\frac{m+1}{2}}   P_{0,t_i}  [V^\beta u ](0,x)   \bigg|
\leq
Kn^{-\frac{m+1}2}.
\label{Eq:CubBound3}
\end{align}
This last inequality is obtained from Lemma \ref{Lem:ApproxSum} with 
$ \psi(t)={C}_{\beta} V^\beta [P_{t,T}g(X_{t}^{0,x})] $: the control on $|\psi|$ readily follows from Proposition \ref{Prop:lin_u_reg}. Moreover, since for each $\beta \in \mc{A}_{M-2}$, we have 
\[ \partial_t E[V^\beta u(t,X_{t})] = \E[ L^0 V^\beta u(t,X_{t})] = \E\left[V^{0\star \beta} u(t,X_{t}) + \sum_{i=1}^d V^{(i,i)\star \beta} u(t,X_{t})  \right],\]
it follows that $\psi$ has well defined locally bounded first order derivatives in $[0,T)$, and hence it is of bounded variation in $[0,T-\epsilon]$ for all $\epsilon >0$ as required. 
\end{enumerate}
\eproof

\section{Study of the Backward Approximation}
\label{Sec:BwdAnalysis}

Our goal in this section is to study the error terms in the approximation of the backward function, given by 
\[   \Delta \hu_i(x) := \hu_i(x) -u(t_i, x).  \]

An inspection of the proof of Theorem \ref{Thm:ExpansionFwdCumulated} shows that we used the linearity properties of the operators $P$ and $Q$ to decompose the global error in a sum of one-step errors. 

We follow a similar idea to expand the error in the case of backward equations.
However, the non-linearity will change the type of decomposition in terms of one-step errors that can be achieved, and will require some additional approximations. 
For this purpose, let us define
\[   \tv_i(x) = \E^{\hQ}\mbra{u(t_{i+1}, \hX_{t_{i+1}}^{t_i,x})\frac{\Delta \homega_i}{h_i}} ,\]
and set
\[ \Delta \tv_i(x) := \hv_{i}(x) - \tv_i(x)  =  \frac{1}{h_i} \E^{\hQ} \mbra{  \Delta \hu_{i+1}(\hX_{t_{i+1}}^{t_i,x} )  \Delta \homega_i  }. \]

Then, we can re-write the decoupling function evaluated at the grid times as the solution of a perturbed scheme, namely

\[ u(t_i,x) = \E^{\hQ}\mbra{u(t_{i+1},\hX^{t_i,x}_{t_{i+1}}) + h_{i} f(x, u(t_i,x), \tv_i(x) )} + \zeta_i(x)  \]
with
 {
\begin{align*}
\zeta_i(x) &:=  \overbrace{\E\mbra{u(t_{i+1},X^{t_i,x}_{t_{i+1}})} -  \E^{\hQ}\mbra{u(t_{i+1},\hat{X}^{t_i,x}_{t_{i+1}})}}^{\zeta_i^e(x)}
  + h_{i} \overbrace{\left\{ f(x,u(t_i,x), \brv_i(x))  - f(x,u(t_i,x), \tv_i(x)) \right\}}^{\zeta_i^f(x) } \\
  & + \underbrace{\E\mbra{\int_{t_i}^{t_{i+1}} \set{f(X_s^{t_i,x}, u(s, X_s^{t_i,x}), v(s, X_s^{t_i,x})) -  f(x,u(t_i,x), \brv_i(x)) }\ud s} }_{\zeta_i^\tau(x)}
\end{align*}
and $\brv_{i}(x) = \E \mbra{\frac{\Delta W_i}{h_i}u(\ti{+1},X^{\ti{},x}_{\ti{+1}})}$. That is, the perturbation is explained by contributions due to the cubature approximation (term $\zeta_i^e$), due to the approximation of the $v$-term in $f$ (term $\zeta_i^f$), and due to discretisation (term $\zeta_i^\tau$).
}
\subsection{One step analysis}
 {
In this section, we use the above decomposition to show that 
\begin{equation}\zeta_i(x) = \varphi(t_i,x)h_i^2+\xi(t_i,x)h_i^3,
\label{Eq:zeta_decompose}
\end{equation}
where $\varphi$ is given explicitly and $\xi$ is a bounded measurable function. More precisely, we show that
\begin{align}
\varphi(t,x) &= \varphi^e(t,x) + \varphi^f(t,x)+\varphi^\tau(t,x) \,,
\label{Eq:varphi_def}
\\
\xi(t,x) &= \xi^e(t,x) + \xi^f(t,x)+\xi^\tau(t,x)\,,
\label{Eq:xi_def}
\end{align}
where each component plays a similar role for the corresponding perturbation term.
}
Let us recall the following regularity result valid under our set of assumptions.
\begin{Proposition}[Space regularity of $u$]  
Suppose either Assumption \ref{H0} or \ref{H1} hold. For all $0<k\leq M$, and $\beta \in \mc A_k$, the function $V^{\beta} u$ is well-defined. Moreover,  
\begin{itemize}
\item Under Assumption \ref{H0}, there exists a constant $C$ such that for all $t<T$  \[ || u(t,.) ||_{k,\infty} \leq C ||g ||_{k,\infty}.\] 
\item Under  Assumption \ref{H1}, there exists a constant C such that for all $t<T$ \[ || u(t,.) ||_{k,\infty} \leq C |g|_{\mathrm{Lip}}(T-t)^{ - \frac{k-1}{2}}.\]
\end{itemize}
\label{Prop:u_reg}
\end{Proposition}

The proposition under Assumption  \ref{H0} is proved by repeated application of the results in \cite{pardoux_backward_1992-1} (in particular Theorem 3.2. and the BSDE representation for the derivative of $u$). 
The proposition under Assumption \ref{H1} is proved in  Theorem 1.4  \cite{crisan_sharp_2012}.

We are now ready to show  \eqref{Eq:zeta_decompose}  and identify the terms in \eqref{Eq:varphi_def}, \eqref{Eq:varphi_def}.

\begin{Lemma} Let $\hQ$ be a cubature measure constructed from a symmetric cubature formula of order $m\geq 3$. Let $i\leq n-2$. 
Under Assumptions \ref{H0} or \ref{H1} with $M\geq 7$ we have that
\begin{enumerate}[i.]
\item $\E\mbra{u(t_{i+1},X^{t_i,x}_{t_{i+1}})} -  \E^{\hQ}\mbra{u(t_{i+1},\hat{X}^{t_i,x}_{t_{i+1}})} = 
  \varphi^e(\ti{},x){h_i^2} +  \xi^{e}(t_i,x)  h_i^{3}$,
  \item $f(t_i,x,u(t_i,x), \brv_i(x))  - f(t_i,x,u(t_i,x), \tv_i(x))  = 
  \varphi^f(\ti{},x)h_i + \xi^f(t_i,x)  h_i^{2}$,
  \item $\E\mbra{\int_{t_i}^{t_{i+1}} \set{f(s,X_s^{t_i,x}, u(s, X_s^{t_i,x}), v(s, X_s^{t_i,x})) -  f(t_i,x,u(t_i,x), \brv_i(x)) }\ud s} 
 = \varphi^\tau(\ti{},x)h_i^2 + \xi^\tau(t_i,x)   h_i^{3}$,
\end{enumerate}
 where $\xi^e,\xi^f,\xi^\tau$ are  bounded measurable functions,
 \begin{align*}
  \varphi^e(t_i,x)& = -\sum_{\beta \in \mathcal{A}_{4} \setminus \mathcal{A}_3} {C}_\beta V^{\beta} u(t_i,x)
  \\
  \varphi^f(t_i,x)& =  \sum_{j=1}^{d} \partial_{z_j} f\mpar{t_i,x,u(t_i,x),  \brv_i(x)} \mbrace{  \sum_{\beta \in \mathcal{A}_{3} \setminus \mc{A}_2 } \mpar{ C_{j\star\beta} + \sum_{k=1}^{|\beta|} C_{\beta_{1:k} \star j \star \beta_{k+1:|\beta|_0}}  }V^{\beta} u(t_i,x) }
  \\
  \varphi^{\tau}(t_i,x)& =    \sum_{\beta \in \mathcal{A}_{2} \setminus \mc{A}_1 } \left( V^{\beta\star(0)} + \frac 12 \sum_{j=1}^d V^{\beta\star(j,j)}  \right) u(t_i,x)  \\
  & \qquad + \sum_{j=1}^d  \sum_{\beta \in \mc{A}_3\setminus \mc{A}_2}  \partial_{z_j} f(t_i,x,u(t_i,x), v(t_i,x)) \mpar{ V^{j\star\beta} + \sum_{k=1}^{|\beta|} V^{\beta_{1:k} \star j \star \beta_{k+1:|\beta|_0}}  } u(t,x)\\ 
  v(t,x) & :=\sum_{k=1}^{d} V^{(k)}u(t,x);
 \end{align*}
the constants $\{{C}_\beta\}_{\beta \in \mc{A}_m}$ are defined as in Proposition \ref{Prop:OneStepCub}. 
Moreover, there exist constants $K,K',K_6, K_7$ such that
  \begin{align*}
  |\xi^e|_{\infty}& \leq K_6 ||u||_{i;6,\infty} +K_7h_i^{1/2} ||u||_{i;7,\infty}\,,
  \\
  |\xi^f|_{\infty} & \leq  K ||u||_{i;5,\infty}\,,
  \\
  |\xi^{\tau}|_{\infty}& \leq  K' ( ||u||_{i;5,\infty} +||u||_{i;6,\infty}).
 \end{align*}
 \label{Lem:OneStepControlBwd}
\end{Lemma}

\begin{Remark}
Under Assumption \ref{H0}, the claim also holds for $i=n-1$. Under Assumption \ref{H1}, ii.\ and iii.\ still hold, while using Lipschitz continuity of $g$, we can replace i.\ by
\begin{itemize}
\item [i'.] $\left | \E\mbra{u(t_{i+1},X^{t_i,x}_{t_{i+1}})} -  \E^{\hQ}\mbra{u(t_{i+1},\hat{X}^{t_i,x}_{t_{i+1}})} \right| \leq  
K {n^{-\frac{\gamma}{2}}} |g|_{Lip} $.
\end{itemize}
\label{Rmrk:LastStepBwd}
\end{Remark}

\begin{Remark}
Note that, if $M$ is large enough 
\begin{equation}
||\varphi(t_i,\cdot)||_{l,\infty} \leq K || u ||_{i;l+4,\infty},
\label{Ineq:ControlVarphi}
\end{equation}
and
\begin{equation}
||\xi(t_i,\cdot)||_{l,\infty} \leq K || u ||_{i;l+6,\infty}.
\label{Ineq:ControlXi}
\end{equation}
\end{Remark}

\noindent\textbf{Proof of Lemma \ref{Lem:OneStepControlBwd}.} 
\begin{enumerate}[i.]
\item  The claim is a consequence of the cubature control in Proposition \ref{Prop:OneStepCub}, and the regularity of $u$ stated in Proposition \ref{Prop:u_reg}.

\item The integration by parts properties of the Stratonovich integral imply that 
\[   \Delta W_i^j J^\beta_{t_i,t_{i+1}} = J^{j}_{t_i,t_{i+1}} J^\beta_{t_i,t_{i+1}} = 
J^{j\star\beta}_{t_i,t_{i+1}} + \sum_{k=1}^{|\beta|} J_{t_i,t_{i+1}}^{\beta_{1:k} \star j \star \beta_{k+1:|\beta|_0}}.\]

Hence, by definition of the symmetric cubature measure and using a Stratonovich Taylor expansion as in the proof of Proposition \ref{Prop:OneStepCub}, we have that
\begin{align}
 \breve{v}^j_{i} (x)  -   \tv_i^j (x)   = &\E \mbra{\frac{\Delta W^j_i}{h_i}u(\ti{+1},X^{t_i,x}_{\ti{+1}})}  -  \E^{\hQ}\mbra{\frac{\Delta \homega^j_i}{h_i}u(\ti{+1},\hX_{t_{i+1}}^{t_i,x})}  \label{Eq:brv-tv}\\ 
  = & h_i \sum_{\beta \in \mathcal{A}_{3}\setminus \mc{A}_2 } V^{\beta} u(\ti{},x) \mpar{ C_{j\star\beta} + \sum_{k=1}^{|\beta|} C_{\beta_{1:k} \star j \star \beta_{k+1:|\beta|_0}}  } + R \notag
\end{align}
 {
where 
\[ R = \sum_{\beta \in \partial \mathcal{A}_m} \E [ \Delta W_i^j J^\beta_{t_i,t_{i+1}}(   V^{\beta} u(.,X^{t_i,x}_.))] -  \E^\QQ [ \Delta \hat w_i^j I^\beta_{t_i,t_{i+1}}(   V^{\beta} u(.,\hat X^{t_i,x}_.))] \]
and we can easily verify that $|R| \leq K_{6} h^{2}_i ||u||_{i;5,\infty}$.} Denoting
\begin{align*}
 \breve{\alpha}_i^{j,k} (x) = \int_0^1 (1-\lambda)\partial_{z_j,z_k}f(t_i,x,u(t_i,x),\brv_i(x)+\lambda(\tv_i(x)-\brv_i(x)))\ud \lambda\,,
\end{align*}
we have
\begin{align*}
 f(t_i,x,  & u(t_i,x), \tv_i(x))  -  f(t_i,x, u(t_i,x),\brv_i(x))  \\
 &=  \sum_{j=1}^d \partial_{z_j}f(t_i,x, u(t_i,x),\brv_i(x)) (\tv_i^j(x)-\brv_i^j(x)) + \sum_{j=1}^d\sum_{k=1}^d \breve{\alpha}^{j,k}_i(x) (\tv_i^k(x)-\brv_i^k(x))(\tv_i^j(x)-\brv_i^j(x))
\end{align*}
and replacing \eqref{Eq:brv-tv} this leads to
\begin{align*}
 f( & x,  u(t_i,x), \tv_i(x))  -  f(x, u(t_i,x),\brv_i(x))  \\
 &=  h_{i}\sum_{j=1}^d \partial_{z_j}f(x, u(t_i,x),\brv_i(x)) \mbrace{ \sum_{\beta \in \mathcal{A}_{3}\setminus \mc{A}_2 } V^{\beta} u(\ti{},x) \mpar{ C_{j\star\beta} + \sum_{k=1}^{|\beta|} C_{\beta_{1:k} \star j \star \beta_{k+1:|\beta|_0}}  } } + \xi^f(t_i,x) h_i^2 \\
\end{align*}

where
\[ | \xi^f |_\infty \leq  K_6  ||u||_{i;5,\infty} .\]

\item  {To prove this claim, we rewrite  $f(t, x, u(t,x), v(t,x)) - f(t, x, u(t,x), v_i(x) \equiv v(t_i,x) ) $,  and  $f(t_i,x, u(t_i,x), \brv_i(x) )- f(t, x, u(t,x), v_i(x)) $ in terms of differential operators over $u$ and then effect the difference. }

To do this, note that the regularity of $u$,  implies that it solves classically the equation
\[ f(t, x, u(t,x), v(t,x)) = L^{(0)} u(t,x)  = \left( V^{(0)} + \frac{1}{2} \sum_{j=1}^d V^{(j,j)}\right) u(t,x).\] 
Hence, Fubini's theorem implies
\begin{align*}
 \E & \mbra {\int_{\ti{}}^{\ti{+1}}  [f\mpar{s,X_s^{t_i,x}, u(s,X_s^{t_i,x}), v(s,X_s^{t_i,x}} - f\mpar{ t_i, x, u(t_i,x),v_i(x)}] \ud s } \\
 & \qquad =    {\int_{\ti{}}^{\ti{+1}}}   \E \mbra { \left( V^{(0)} + \frac 1 2 \sum_{j=1}^d V^{(j,j)}\right) u(s,X_s^{t_i,x}) - \left( V^{(0)} + \frac 12 \sum_{j=1}^d V^{(j,j)}\right) u(t_i,x)}  \ud s\\
 & \qquad =    h_i^2 \sum_{\beta\in \mc{A}_2 \setminus \mc{A}_1} V^{\beta}\left( V^{(0)} + \frac 12\sum_{j=1}^d V^{(j,j)}\right) u(t_i,x) +  R_i(x)h_i^3.
\end{align*}
where
$|R_i|_\infty < K||u||_{6,\infty}$. Moreover,
\begin{align} 
v^j(t_i,x) -\brv_i^j(x) = & \mbrace{ V^{j}u(t_i,x) - \E\mbra{\frac{\Delta W^j_i}{h_i} u(t_i, X^{t_i,x}_{t_i+1}) } } \notag \\
= & - h_i\sum_{\beta \in \mc{A}_3\setminus \mc{A}_2} \mpar{ V^{j\star\beta} + \sum_{k=1}^{|\beta|} V^{\beta_{1:k} \star j \star \beta_{k+1:|\beta|_0}}  } u(t,x) + R^v_i(x) h_i^{2}  
\label{Eq:v-vbreve}
\end{align}
with \[|R^v_i|_\infty \leq \# \left(  \partial \mc{A}_{3} \cap \mc{A}_5    \right)  \cdot    ||u||_{i;5,\infty} . \]
 {with $\partial A_3$ defined as in  section \ref{subsec:multiindex}}. 

Turning now to $ f(t_i,x, u(t_i,x), \brv_i(x) )$, we conclude by a similar argument as before that
\begin{align*}
 f(t_i,x, u(t_i,x), \brv_i(x) )  = &  f(t_i,x,u(t_i,x),v(t_i,x)) \\
 & +  h_{i} \sum_{j=1}^d  \sum_{\beta \in \mc{A}_3\setminus \mc{A}_2}  \partial_{z_j} f (t_i,x) \mpar{ V^{j\star\beta} + \sum_{k=1}^{|\beta|} V^{\beta_{1:k} \star j \star \beta_{k+1:|\beta|_0}}  } u(t,x) \\
  & +  R'_i(x) h_i^2.
\end{align*}
where we wrote $\partial_{z_j} f(t_i,x)$ to mean the quantity $\partial_{z_j} f(t_i,x,u(t_i,x), v(t_i,x))$, and with $|R_i'|_\infty < ||\partial_{z_j} f ||_\infty |R^v_i|_\infty || u ||_{i;5,\infty}$. The claim then follows. \eproof
\end{enumerate}

\vspace{5pt}

\subsection{Error expansion}

We now proceed to show the full error expansion under our set of assumptions.
Let us recall the following result on the rate of convergence of the scheme (see  \cite{crisan_solving_2012}  or \cite{chaudru_de_raynal_cubature_2015}).

\begin{Lemma}
Let $\hQ$ be a cubature measure constructed from a symmetric cubature formula of order $m\geq 3$. 
Under Assumptions \ref{H0} or \ref{H1} with $M\geq 6$, (and $\gamma\geq 2$ in the case of Assumption $\ref{H1}$), we have
\[|\Delta \hu_i|_{\infty}   \leq K \left( | \Delta \hat{u}_{n-1}|_{\infty}+  \sum_{j=i}^{n-2} ||u(t_j,.)||_{4,\infty} h_j^2 + ||u(t_j,.)||_{6,\infty} h_j^3 \right) \leq K' n^{-1}.\]
Moreover,
\[|\Delta \hu_i|_{\infty}^2 +  \sum_{j=i}^{n-1}h_j |\Delta \tv_j|_{\infty}^2  \leq K \left(|\Delta \hat{u}_{n-1}|_{\infty}^2 + \sum_{j=i}^{n-1} ||u(t_j,.)||_{4,\infty} h_j^2 + ||u(t_j,.)||_{6,\infty} h_j^3 \right)^2 \leq K' n^{-2}.\]
\label{Lem:GlobalControlBackward}
\end{Lemma}

\subsubsection{One-step expansion}

We first prove the following key one-step expansion.

\begin{Lemma}
Under the assumptions of Lemma \ref{Lem:OneStepControlBwd}, we have that

\begin{align}
  \Delta \hu_i(x)  = &  \E^{\hQ} \mbra{   \hat{\Lambda}_{i,i+1}^x \Delta \hu_{i+1}(\hX^{t_i,x}_{t_{i+1}})} 
  - e^{h_i \eta^y(t_i,x) } {\varphi} (t_i,x) h_i^2  -e^{h_i \eta^y(t_i,x) } \xi(t_i,x) h_i^3           \label{Eq:DeltaUHat}\\
  & \quad   -  e^{h_i \eta^y(t_i,x) }  R^{\exp}_i(x) \Delta \hu_i(x)  +  h_i  e^{h_i \eta^y(t_i,x) } R_i^{u,\tv}(x) \,, \notag 
\end{align}
where
\begin{align}
 \eta^y(t_i,x) &:= \partial_y f(t_i,x,u(t_i,x),\tv_i(x))\,, \qquad \eta^z(t_i,x) := \partial_zf(t_i,x,u(t_i,x),\tv_i(x))\,, \nonumber
\\
\hat{\Lambda}_{i,j}^x &:=  \prod_{k=i}^{j-1} \left[  e^{h_k \eta^y (t_k,\hX_{t_{k}}^{t_i,x}) } (1+ \eta^z(t_k,\hX_{t_{k}}^{t_i,x}) \Delta \homega_k) \right] \,, \text{ for }j> i\,, \label{eq de hatLambda}
\end{align}
 and  $ \hat{\Lambda}_{i,i}^x:= 1$. 
 Moreover, for some constants $C^z,C^{\exp},C^{u}$, $C^{\tv}$
\begin{align*}
&|R_i^z|_{\infty} \leq  C^{z} |\Delta \hu_{i+1}|_\infty h_i^2\,, \qquad    |R_i^{\exp}|_{\infty} \leq  C^{\exp} h_i^2\,, \\
& \text{ and }\;|R_i^{u,\tv}|_\infty \leq C^{u} |\Delta \hu_i|^2_\infty + C^{\tv} |\Delta \tv_{i}|^2_\infty  . 
\end{align*}
\label{Lem:ErrorExpansionOneStep}
\end{Lemma}
\proof
Since the exact solution satisfies a perturbed version of the scheme, we have
\begin{align*}
  \Delta \hu_i(x) = & \E^{\hQ} \mbra{ \hu_{i}(\hX^{t_i,x}_{t_{i+1}}) -  u(t_{i+1},\hX^{t_i,x}_{t_{i+1}})}  -  \varphi(t_i,x) h_i^2 - \xi(t_i,x) h_i^3 \\
 & + h_i \left[ f(t_i,x,\hu_i(x), \hv_i(x)) - f(t_i,x,u(t_i,x),\tv_i(x)) \right]  . 
\end{align*}

We compute, using the mean value theorem, 
\begin{align*}
 f(t_i,x,\hu_i(x), & \hv_i(x)) - f(t_i,x,u(t_i,x),\tv_i(x))  = \eta^y(t_i,x)  \Delta \hu(t_i,x) 
    + \eta^z(t_i,x)   \Delta \tv_i(x) + R_i^{u,\tv}(x)
 \end{align*}
 where
 \begin{align*}
  R_i^{u,\tv}(x) = &    \alpha^{yz}_i  \Delta u_i(x)  \Delta \tv_i(x) 
    +  \frac12 \alpha^{yy}_i  | \Delta \hu_i(x)|^2   +  \frac12 \alpha^{zz}_i | \Delta \tv_i(x)|^2   \\
\leq & C^u  | \Delta \hu_i(x)|^2 + C^{\tv}  | \Delta \tv_i(x)|^2 \\
\end{align*}
for some $\alpha_i^{yz}, \alpha_i^{xy}, \alpha_i^{zz}$ with norm bounded by $\sup\limits_{|\beta|_1 \leq 2} |D^\beta f| $. 

This leads to
\begin{align}
  \Delta \hu_i(x)  \left(1-h_i \eta^y(t_i,x) \right)  = \E^{\hQ} & \mbra{  \Delta \hu_i (\hX^{t_i,x}_{t_{i+1}}) \left(1+ \Delta \homega_i  \eta^z(t_i,x)  \right) }  - \varphi(t_i,x)h_i^2 \label{Eq:Deltahu1} \\ 
  &  - \xi(t_i,x) h_i^3  + h_i R_i^{u,\tv}(x) \,. \notag 
\end{align}

On the other hand, notice that for $h_i$ small enough,
\begin{align}
\frac{1}{1-h_i \eta^y(t_i,x) }  = e^{h_i  \eta^y(t_i,x) } + R_i^{\exp}(x) 
\label{Eq:ExpEtaY}
\end{align}
where
\[ |R_i^{\exp}|_\infty \leq K h_i^2|\partial_y f|_{\infty} \leq C^{\exp} h_i^2 .\]

The claim follows by inserting \eqref{Eq:ExpEtaY} and   the definition of $\hat{\Lambda}_{i,i+1}^x$  in \eqref{Eq:Deltahu1}.
\eproof

\subsubsection{Global Expansion}

\begin{Lemma}
\[ \left| \Delta \hat u_i (x) - \Delta \hat u_{n-1} (x) + \sum_{j=i}^{n-2} h_j^2 Q_{t_i,t_j} \left[ \hat{\Lambda}_{i,j}^{x} e^{h_j\eta^y(t_j,.)}    \varphi(t_j, .) \right](x) \right| \leq K n^{-2}\,,     \]
where $\hat{\Lambda}^x$ is defined in \eqref{eq de hatLambda}.

\label{Lem:PreGlobal}
\end{Lemma}

\proof
Let us start by finding some controls on the elements of the sum in the statement. 

By independence of  increments, assuming $ i+2\leq j \leq n-2$ we compute
\begin{align*} 
\E \left[  \prod_{k=i}^{j-1} \left ( 1+ \eta^z(t_k,\hX_{t_{k}}^{t_i,x}) \Delta \homega_k \right)^2 \right]  &\leq 
\E \left[  \prod_{k=i}^{j-2} \left(1+ \eta^z(t_k,\hX_{t_{k}}^{t_i,x}) \Delta \homega_k \right)^2 \right]
(1+ |\eta^z(t_{j-1},\cdot)|^2_\infty h_{j-1}),
\end{align*}
so that a straightforward induction implies, for $ i\leq j \leq n-2$
\begin{align*}
\E \left[  \prod_{k=i}^{j-1} \left ( 1+ \eta^z(t_k,\hX_{t_{k}}^{t_i,x}) \Delta \homega_k \right)^2 \right]
& \leq \exp( \sum_{k=i}^{j-1} |\eta^z(t_k, \cdot)|_\infty^2 h_k ) \;.
\end{align*}
Hence, by the Cauchy-Schwarz theorem
\[ \E^{\hQ} \left[|\hat{\Lambda}_{i,j}^x|\right] \leq \exp\left( | \eta^y   |_\infty  T + \frac{1}{2} |\eta^z|_\infty^2 T  \right)  < \infty .\]

Now,  the flow property of the cubature approximation, i.e.
\[ \E^{\Q}\left[\psi\left(\hX_{t_k}^{t_j,\hX_{t_j}^{t_i,x}}\right)\right] =  \E^{\Q}\left[\psi(\hX_{t_k}^{t_i,x})\right]\]
for any measurable function $\psi$, implies that  we get, by iterating over $i$ the result in Lemma \ref{Lem:ErrorExpansionOneStep}, that
\begin{align*}
  \Delta \hat u_i(x)  - \Delta \hat u_{n-1} (x) & = -\E^{\Q} \left[  \sum_{j=i}^{n-2} \hat{\Lambda}_{i,j}^x \left(  e^{h_j \eta^y(t_j,\hX_{t_j}^{t_i,x} ) } {\varphi} (t_j, \hX_{t_j}^{t_i,x}) h_j^2   \right) \right] \notag \\
  & \quad+  \E^{\Q} \left[  \sum_{j=i}^{n-2} \hat{\Lambda}_{i,j}^xe^{h_j \eta^y(t_j, \hX_{t_j}^{t_i,x}) } \left(  R^z_j( \hX_{t_j}^{t_i,x})  -  \xi(t_j,\hX_{t_j}^{t_i,x}) h_j^3   \right) \right] \\
 & \quad +  \E^{\Q} \left[  \sum_{j=i}^{n-2} \hat{\Lambda}_{i,j}^x  e^{h_j \eta^y(t_j, \hX_{t_j}^{t_i,x}) } \left(  h_j  R_j^{u,\tv}(\hX_{t_j}^{t_i,x})  - R^{\exp}_j(\hX_{t_j}^{t_i,x}) \Delta \hu_i   \right) \right].
\end{align*}

Hence,  
\begin{align*}
& \left| \Delta \hat u_i (x)-  \Delta \hat u_{n-1} (x) + \sum_{j=i}^{n-2} h_j^2 Q_{t_i,t_j} \left[ \hat{\Lambda}_{i,j}^x e^{h_j\eta^y(t_j, .)}    \varphi(t_j, .) \right](x) \right| \\
&  \qquad \leq   \E^{\Q} \left[  \sum_{j=i}^{n-2} \hat{\Lambda}_{i,j}^x e^{h_j \eta^y(t_j, \hX_{t_j}^{t_i,x}) }  \left(   ||u ||_{j;6,\infty} h_j^3  + C^{\exp} h_j^2 | \Delta \hu_j|      \right) \right]\\
 & \qquad \quad +  \E^{\Q} \left[  \sum_{j=i}^{n-2} \hat{\Lambda}_{i,j}^x  e^{h_j \eta^y(t_j, \hX_{t_j}^{t_i,x}) } \left(  h_j (C^u|\Delta \hu_j|_\infty^2+C^{\tv}|\Delta \tv_j|_\infty^2 )     \right) \right]\\
 &  \qquad \leq    \sum_{j=i}^{n-2}    C' ||u ||_{j;6,\infty} h_j^3     +  \sum_{j=i}^{n-2} C_2^{\exp} h_j^2 n^{-1} \\
 & \qquad \quad +  \sum_{j=i}^{n-2}  h_j^2 C^u_2 n^{-2}+  \sum_{j=i}^{n-2}  h_jC_2^{\tv}|\Delta \tv_j|_\infty^2 )           \\
 & \qquad \leq K n^{-2}
\end{align*}
where we have used, the controls in Lemma \ref{Lem:ErrorExpansionOneStep} and the rate of convergence results in Lemma \ref{Lem:GlobalControlBackward} for the first inequality,  controls on the expectation of $\hat{\Lambda}_{i,j}^xe^{h_j \eta^y(t_j, \hX_{t_j}^{t_i,x})} $ on the second inequality and Lemma \ref{Lem:GlobalControlBackward}  and Corollary \ref{Cor:ControlSum2} for the last inequality. \\
\eproof\\

Although in Lemma \ref{Lem:PreGlobal} we have identified (up to the last step) an explicit coefficient for the main error term,  it still depends on the actual approximation algorithm. 
In order to have a more tractable expression, let us re-express the result of Lemma \ref{Lem:PreGlobal}  in terms of a family of linear operators forming a semigroup. 
Let us introduce
\begin{equation}
 \hat{\Theta}_{t_i,t_j} \psi(x) :=  Q_{t_i,t_j} \left[\hat \Lambda^x_{i,j} \psi(.) \right] (x) ; \quad t_i\leq t_j.
 \label{Eq:ThetaHatDef}
\end{equation} 
where $ \hat\Lambda^x_{i,j} $ is defined in \eqref{eq de hatLambda}.

Indeed, for any $i,j,k \in \{0, \ldots, n\}$ with $i\leq j\leq k$ we can verify the associativity property 
\[  \hat{\Theta}_{t_i,t_j}  \left[   \hat{\Theta}_{t_j,t_k} \psi(.)  \right] (x)  = Q_{t_i,t_j} \left[\hat \Lambda^x_{i,j} Q_{t_j,t_k} \left[\hat \Lambda^x_{j,k} \psi(.) \right] \right] (x) = Q_{t_i,t_k} \left[\hat \Lambda^x_{i,k} \psi(.) \right] (x) = \hat{\Theta}_{t_i,t_k} \psi(x)   \]
and its linearity.

We show in the following that this operator can be seen as an approximation of 
\begin{equation}
\Theta_{t_i,t_j} \psi(x) :=  P_{t_i,t_j} \left[ \Lambda_{t_i,t_j}^{x} \psi(.) \right] (x) ,
\label{Eq:ThetaDefined}
\end{equation} 
where, for $r \le t$,
\begin{align}
 \Lambda_{r,t}^{x} = & \exp  \left[ \int_{r}^t \eta^z(s,X_s^{r,x})   \ud W_s  + \int_{r}^t  \eta^y(s,X_s^{r,x})- \frac{1}{2} |\eta^z(s,X_s^{r,x})|^2  \ud s    \right], \label{eq de Lambda} \\
 = & \exp  \left[ \int_{r}^t \eta^z(s,X_s^{r,x})  \circ \ud W_s  - \frac{1}{2} \int_{r}^t \left( 2 \eta^y(s,X_s^{r,x}) + |\eta^z(s,X_s^{r,x})|^2 +  \sum_{a=1}^n V^{(a)}\eta^z_a(s,X_{s}^{r,x}) \right)  \ud s   \right], \nonumber
\end{align}
where $\eta^y$ and $\eta^z$ are defined in \eqref{eq de hatLambda}.
Clearly, the family  $(\psi \mapsto \Lambda_{r,t}^{x}\psi)_{r,t}$ also forms a linear semigroup of operators.

\begin{Lemma}
Assume that $m\geq 3$, that the cubature measure is symmetric and that  $\psi \in \bar{C}_b^5$. Then,
\[ |\Theta_{t_k,t_{k+1}} \psi(x) - \hat{\Theta}_{t_k,t_{k+1}} \psi(x) | \leq h_k^2 ||\psi||_{4,\infty} + h_k^{5/2} ||\psi||_{5,\infty} .  \]
\label{Lem:OneStepTheta}
\end{Lemma}
\begin{proof}

 {We look for an Itô-Taylor expansion for the two terms. From the definition of $\Theta$ in \ref{Eq:ThetaDefined} and $\hat \Theta$ in  \eqref{Eq:ThetaHatDef}, we see that we need to consider the joint dynamics of $X$, $\Lambda$  and $\hat \Lambda$. Consider the system}
\[  d\begin{bmatrix}  X^{t_k,x}_t\\  {\varrho^{t_k,x}_t} \\ \chi_t^{t_k,x} \end{bmatrix} 
= \begin{bmatrix} \bar{b}(t,X_{t}^{t_k,x}) \\  
- \frac{1}{2} \varrho^{t_k,x}_t \left(  |\eta^z(t,X_t^{t_k,x})|^2 + \sum_{a=1}^d V^{(a)}\eta^z_a(t,X_{t}^{t_k,x}) \right) \\ 
\chi_t^{t_k,x} \left(\eta^y(t,X_t^{t_k,x})- \eta^y(t_k,x) \right)  
\end{bmatrix} \ud t  
+\begin{bmatrix} \sigma(t,X_t^{t_k,x}) \\ 
\varrho^{t_k,x}_t \eta^z(t,X_t^{t_k,x}) \\ 
0 \end{bmatrix} \circ \ud W_t \]
with initial condition $(X^{t_k,x}_{t_k}, \varrho^{t_k,x}_{t_k}, \chi^{t_k,x}_{t_k} )^\top = (x,1,e^{h_k \eta^y(t_k,x)})^\top$.   { Note in particular that $\Lambda^x_{t_k,t} = \varrho^{t_k,x}_t \chi_t^{t_k,x}$, while $\hat \Lambda^x_{t_k,t} =  \chi_t^{t_k,x}(1+\eta^z(t_k,\hat X_{t_k}^{t_i,x}) \Delta \hat \omega_k) $.}

Let us denote by $\tilde V^\beta$ the  iterated differential operators associated to this system,  {so that we have for any sufficiently regular $\phi: [0,T]\times\R^d\times\R\times\R\rightarrow \R$,} 
\begin{align*}  
\tilde{V}^{(a)} \phi(t,y,\rho,\chi)   = & V^{(a)} \phi(t,y,\rho, \chi) + \rho \eta_a^{z}(t,y)\partial_\rho \phi(t,y,\rho, \chi), \text{ for } a \in \set{1,\dots,d}\,; \\
 \tilde{V}^{(0)} \phi(t,y,\rho, \chi)   = & V^{(0)} \phi(t,y,\rho, \chi) - \frac{1}{2} \rho \left( |\eta^z(t,y)|^2 +  \sum_{a=1}^n V^{(a)}\eta^z_a(t,y) \right)\partial_\rho \phi(t,y,\rho, \chi) \\
& \quad  + \chi \left(\eta^y(t,y)  - \eta^y(t_k,x)\right) \partial_\chi \phi(t,y,\rho, \chi)  ; \\
 \tilde{V}^{(a_1,a_2)} \phi(t,y,\rho,\chi)   = & V^{(a_1,a_2)} \phi(t,y,\rho, \chi) + \rho   V^{(a_1)} [ \eta_{a_2}^{z}](t,y)\partial_\rho \phi(t,y,\rho, \chi)\\
 & +   \rho \eta_{a_2}^{z}(t,y)   V^{(a_1)} [ \partial_\rho \phi](t,y,\rho, \chi) + \rho \eta_{a_1}^{z}(t,y)   V^{(a_2)} [ \partial_\rho \phi](t,y,\rho, \chi) \\
 & + \rho \eta_{a_1}^{z}(t,y)  \eta_{a_2}^{z}(t,y)  \partial_\rho \phi(t,y,\rho, \chi) + \rho^2 \eta_{a_1}^{z}(t,y)  \eta_{a_2}^{z}(t,y)  \partial^2_{\rho^2 }\phi(t,y,\rho, \chi),
 \end{align*}
for $a_1, a_2 \in \set{1,\dots,d}$. Thus, from a Taylor-Stratonovich expansion of order 3 on the function $\phi(t,y,\rho,\chi)=\rho \chi \psi(y)$, and after some calculations we obtain
\begin{align*}  
\Theta_{t_k,t_{k+1}} \psi(x) & = \E[  \varrho^{t_k,x}_{t_{k+1}}   \chi^{t_k,x}_{t_{k+1}} \psi(t_{k+1}, X^{t_k,x}_{t_{k+1}}) ] \\
= & \sum_{\beta \in \mathcal{A}_3} (\E J^\beta_{t_k, t_{k+1}}) \tilde V^{\beta} \phi (t_k,x,e^{\eta^y(t_k,x)},1)  + \sum_{\beta \in \partial \mathcal{A}_3} \E J^\beta_{t_k,t_{k+1}}(   \tilde V^{\beta} \phi(.,X^{t_k,x}_., \varrho^{t_k,x}_{.},   \chi^{t_k,x}_{.} )  )\\
 = &  e^{ \eta^y(t_k,x) h_k} \psi(t_k,x) +   h_k e^{\eta^y(t_k,x)h_k} \left[\right(V^{(0)} + \frac{1}{2}\sum_{a=1}^d V^{(a,a)} + \eta^z_a V^{(a)}\left) \psi\right] (t_{k},x)  \\
 & + \frac{1}{2} h_k \sum_{a=1}^d  e^{\eta^y(t_k,x)h_k} (\psi  V^{a}[ \eta^z_{a} ]  ) (t_k,x) +  \sum_{\beta\in \partial \tilde{ \mc{A}_3}} \E \left( J^\beta_{t_k,t_{k+1}} \left[ \tilde{V}^\beta \phi \right] \right).
\end{align*}
 {
Similarly, we can define $\hat \phi_1(t,y,\rho,\chi) = \chi \psi(t,x)$, $\hat \phi_2^a (t,y,\rho,\chi) = \chi \eta^z_a(t,x)\psi(t,x)$ for |$i=1, \ldots, d$ and repeat the development on $\hat \phi := \hat \phi_1(t,y,\rho,\chi) + \sum_{a=1}^d \phi_2^a (t,y,\rho,\chi) \Delta \hat w_a$. 
}
We get   
\begin{align*}  
\hat \Theta_{t_k,t_{k+1}} \psi(x) & = \E^{\hQ}[ (1+\eta^z(t,x) \Delta\homega_k)  e^{\eta^y(t_k,x)h_k} \psi(t_k,\hX_{t_{k+1}}^{t_k,x} )  ] \\
 = &  e^{ \eta^y(t_k,x) h_k} \psi(t_k,x) +     h_k e^{\eta^y(t_k,x)h_k} \left[\right(V^{(0)} + \frac{1}{2}\sum_{a=1}^d V^{(a,a)} + \eta^z_a V^{(a)}\left) \psi\right] (t_{k},x)   \\
 & +  \frac{1}{2} h_k \sum_{a=1}^d  e^{\eta^y(t_k,x)h_k} (\psi  V^{a}[ \eta^z_{a} ]  ) (t_k,x) +  \sum_{\beta\in \partial \tilde{ \mc{A}_3}} \E^{\hQ} \left( I^\beta_{t_k,t_{k+1}} \left[ \tilde V^\beta \hat \phi_1 \right] \right)\\
 & \sum_{a=1}^d \sum_{\beta\in \partial \tilde{ \mc{A}_3}} \E^{\hQ} \left(   I^{a\star \beta}_{t_k,t_{k+1}} \left[ \tilde V^{a\star \beta} \hat \phi_2 \right] +    \sum_{\iota = 1}^{|\beta|_0} I^{\beta_{1:\iota}\star a\star\beta_{\iota+1:|\beta|_0}   }_{t_k,t_{k+1}} \left[ \tilde V^\beta \hat \phi_2  \right] \right) .
\end{align*}

 {
Note that all terms match except for the residuals, which we can control as Proposition \eqref{OneStepCub}}.\eproof
\end{proof}

\begin{Proposition}\label{pr interm nonlin exp 1}
Suppose that $M\geq 9$. For every $i<j \in \set{1, \dots, n}$, assume that $\gamma \geq m\geq 3$ and that the cubature measure is symmetric. Then, for $\varphi$ defined in \eqref{Eq:varphi_def},

\[  \left|  \sum_{j=i}^{n-2} h_j^2 \left\{  \Theta_{t_i,t_j} \left[\varphi(t_j, .) \right](x)  - \hat \Theta_{t_i,t_j} \left[ \varphi(t_j, .) \right](x) \right\}  \right| \leq K n^{-2} .\]
\end{Proposition}
\proof

Let us write for convenience 
\[\nu(t_k,x;t_j)   =  \Theta_{t_k,t_j} \left[ \varphi(t_j,.)\right] (x).\]
We first show that, 
\begin{equation}  |V^\beta \nu (t_k,x;t_j) | \leq  K ||\varphi (t_k,.)||_{\beta,\infty}  \leq K' || u(t_k,.) ||_{\beta+4,\infty}.
\label{Eq:NuGradientBounds}
\end{equation}
Recall that $\varphi(t_k,.)$ is, essentially, a sum of derivatives of $u(t_k,.)$ multiplying bounded functions. Hence, broadly speaking, the first inequality tells us that the operator regularizes as if we could ``extract derivatives from the operator''. The second inequality is a direct consequence of  {the explicit expressions for $\varphi$ in Lemma \ref{Lem:OneStepControlBwd} and the  regularity properties of $f, b, \sigma$}.

It is then sufficient to prove the first inequality in \eqref{Eq:NuGradientBounds}. Since
\begin{align*}
\nu(t_k,x;t_j)  &  = \E \left[  e^{ \int_{t_k}^{t_j} \eta^z(s,X_s^{t_k,x}) \ud W_s  +\int_{t_k}^{t_j} \{ \eta^y(s,X_s^{t_k,x}) -\frac{1}{2} |\eta^z(s,X_s^{t_k,x})|^2  \}\ud s   } \varphi( t_j,X_{t_j}^{t_k,x} )  \right]\\
&  = \E^{\P^*} \left[  e^{  \int_{t_k}^{t_j}  \eta^y(s,X_s^{t_k,x})\ud s } \varphi(t_j, X_{t_j}^{t_k,x} ) \,, \right]
\end{align*}
for an appropriately defined $\P^*$. Given our assumptions, Girsanov's theorem implies that the diffusion $X^{t_k,x}$ solves the SDE
\[  X_t^{t_k,x} = x + \int_{t_k}^t \left\{ b(s,X_s^{t_k,x}) + \sigma(s,X_s^{t_k,x}) \eta^z(s, X_t^{t_k,x})^\top \right\}  \ud s +  \int_{t_k}^{t} \sigma(s,X_s^{t_k,x}) \ud W^*_s ,\]
for a $\P^*$ Brownian motion $W^*$, which also satisfies the structural conditions in Assumption \ref{H0} (resp. \ref{H1}). 

Now, note that  because $\eta^y\in \mathcal{C}_b^{N-2}$,  we can show that $ e^{  \int_{t_k}^{t_j}  \eta^y(s,X_s^{t_k,x})\ud s }$ is a Kusuoka-Stroock function as defined in  Definition 22 in \cite{crisan_cubature_2013}. This means that the Malliavin integration by parts results hold up to multiplying by another Kusuoka-Stroock functions (Section 2.6 in \cite{crisan_cubature_2013}). We can then adapt the arguments in Corollary 32 in \cite{crisan_cubature_2013} to deduce the first inequality in \eqref{Eq:NuGradientBounds}. 

Using a telescopic sum,  Lemma \ref{Lem:OneStepTheta} and using \eqref{Eq:NuGradientBounds}, we conclude 
\begin{align*}
& \left|  \sum_{j=i}^{n-2} h_j^2 \left\{  \Theta_{t_i,t_j} \left[\varphi(t_j, .) \right](x)  - \hat \Theta_{t_i,t_j} \left[ \varphi(t_j, .) \right](x) \right\}  \right| \\
& \qquad = \left| \sum_{j=i}^{n-2} h_j^2 \sum_{k=i}^{j-1}  \hat \Theta_{t_i,t_k} \left\{  [ \hat \Theta_{t_k,t_{k+1}} - \Theta_{t_k,t_{k+1}} ]\nu(t_{k+1},.;t_j)   \right  \} (x) \right|  \\
& \qquad \leq  \sum_{j=i}^{n-2} h_j^2 \sum_{k=i}^{j-1} \left|  [ \hat \Theta_{t_k,t_{k+1}} - \Theta_{t_k,t_{k+1}} ]\nu(t_{k+1},.;t_j)   \right|_{L^\infty}    \\
& \qquad \leq  K \sum_{j=i}^{n-2} h_j^2  \left( \sum_{k=i}^{j-1} h_k^2 ||\nu(t_{k+1},.;t_{j})||_{4,\infty} +  h_k^{5/2} ||\nu(t_{k+1},.;t_{j})||_{5, \infty}  \right)   \\
& \qquad \leq  K \sum_{j=i}^{n-2} h_j^2  \left(  \sum_{k=i}^{j-1} h_k^2 ||u(t_{k+1},.)||_{8,\infty}   +  h_k^{5/2} ||u(t_{k+1},.)||_{9,\infty}   \right)   
\\
& \qquad \leq  K \frac{\gamma}{n} \sum_{j=i}^{n-2} h_j^2  \left(  (T-t_j)^{-3/2(1-1/\gamma)}  \right)   
\\
& \qquad \leq K n^{-2},
\end{align*}
where we used Lemma \ref{Lem:ControlSum} to get the last inequality. \eproof

\vspace{5pt}

Before proceeding with the proof of the main expansion result, let us  give the shape of the leading coefficient appearing in the statement of the theorem.

\begin{Definition}\label{de psi-nonlin}
\begin{enumerate}[i.]
\item If Assumptions \ref{H0}  holds and a uniform discretization is used,
\begin{align*}
 \Psi_{T}^{nl}(0,x) =  T \int_0^T\esp{\Lambda_{0,t}^{x} \varphi(t,X_t) } \ud t\,;
\end{align*}
\item If Assumptions \ref{H0}  holds and a decreasing step discretization with $\gamma>m+1$ is used,
\begin{align*}
 \Psi_{T}^{nl}(0,x) =  T^{1/\gamma}\gamma  \int_0^T\esp{\Lambda_{0,t}^{x} \varphi(t,X_t) \mpar{ T - t  }^{\frac{\gamma-1}{\gamma}} } \ud t\,;
\end{align*}
where $\Lambda_{0,t}^{x}$ is defined in \eqref{eq de Lambda}.
\end{enumerate}
\end{Definition}

\noindent\textbf{Proof of Theorem \ref{Thm:MainExpansionBwd}.} 

From Lemma \ref{Lem:PreGlobal},  Proposition \ref{pr interm nonlin exp 1} and equation \eqref{Eq:ThetaDefined} we get

\[ \left| \Delta \hat u_i (x)  +   \sum_{j=i}^{n-2} h_j^2   P_{t_i,t_j} \left[ \Lambda_{t_i,t_j}^{x} \varphi(t_j,.) \right](x)
 \right| \leq K n^{-2}  +  |\Delta\hat u_{n-1}|_\infty \leq K' n^{-2}  \ ,,     \]
where we have used the control on the last step highlighted in Remark \ref{Rmrk:LastStepBwd}.

To conclude our claim, we invoke Lemma \ref{Lem:ApproxSum} with 
$\psi(t)= \esp{\Lambda_{0,t}^{x} \varphi(t,X_t) }$. Indeed, the control on $|\psi|$ readily follows from \eqref{Ineq:ControlVarphi} and the expression for $\Lambda_{0,t}$. Similarly, using the chain rule and the equality
\[ \partial_t \E[V^\beta u(t,X_{t})] = \E\left[V^{0\star \beta} u(t,X_{t}) + \sum_{i=1}^d V^{(i,i)\star \beta} u(t,X_{t})  \right],\]
we conclude that $\psi$ has well defined locally bounded first order derivatives in $[0,T)$, and hence it is of bounded variation in $[0,T-\epsilon]$ for all $\epsilon >0$. 
\eproof

\section{Complexity reduction}
\label{se complexity}

In this section, we aim to control the complexity growth on both the number of steps $n$ and the dimension of the problem.
\subsection{Interpolation operators and numerical schemes}
\label{subse interpolator}
Let $A= \prod_{\ell=1}^d [a^\ell,b^\ell] $ be a hypercube in $\R^d$. Let $\mathbb{E}=\cC^0(A)$ be the set of real valued continuous functions on $A$.  We let $\mathbb{X}(A)$ be a set of interpolation points in $A$. In this section, we consider interpolation operators
$\mathfrak{P}^A: \mathbb{E} \rightarrow \mathbb{E}$
that will be combined with the cubature scheme. To be useful in practice, these operators must satisfy some stability and approximation properties that we explicit now. To this end, we introduce two vector subspaces of $\mathbb{E}$, namely $\mathscr{R}^a$ (used to state the approximation property) with norm $\vvvert . \vvvert$ and $\mathscr{R}^s$ (used to state the stability property). We say that $(\mathfrak{P}^A,\mathscr{R}^a,\mathscr{R}^s)$ is  {compatible} 
if the following conditions are satisfied:
\begin{enumerate}
\item $\mathscr{R}^a \subset \mathscr{R}^s$ and $\mathfrak{P}^A(\mathscr{R}^s) \subset \mathscr{R}^s$
(we can thus restrict $\mathfrak{P}^A$ to $ \mathscr{R}^s$). 
\item
There exists two bounded functions $(\eta_A,\epsilon_A): \mathbb{N}_+ \rightarrow \R_+ \times \R_+ $ such that
\begin{align}
& | \mathfrak{P}^A \psi|_\infty \le e^{\eta_A(\sharp \mathbb{X}(A))}|\psi|_\infty\; \text{ for } \; \psi \in \mathscr{R}^s
\label{eq stability projection}
\\
& |(I-\mathfrak{P}^A )\psi|_\infty \le \epsilon_A(\sharp \mathbb{X}(A)) \vvvert \psi \vvvert \; \text{ for } \; \psi \in \mathscr{R}^a
\label{eq consistency projection}
\end{align}
\end{enumerate}
We now describe the backward scheme that will be used in practice. To simplify the presentation, we will assume that one is interested only in approximating the value function at time $0$ at a given point $x_0 \in \R^d$ and we thus set $A_0 := \set{x_0}$. As the terminal condition is known, no approximation is needed at the last step $n$, we can thus set $A_n = \R^d$ and $\mathfrak{P}_n=\mathfrak{I}$, where $\mathfrak{I}$ is the identity operator on $\cC^0(\R^d)$.
For step $i \in \set{1, \dots, n-1}$, we consider a sequence of hypercubes $(A_i)$ and denote  the associated projection operator $\mathfrak{P}_i:=\mathfrak{P}^{A_i}$. The sequence of hypercubes satisfies, for $1 \le i \le n$,

\begin{align*}
 \bigcup_{x \in \mathbb{X}(A_i) } \mathop{supp}(\hat{X}^{t_{i},x}_{t_{i+1}}) \subset A_{i+1},\quad 0\le i \le n-1 \;.
\end{align*}

\begin{Remark} \label{re const A_i}
\begin{enumerate}
\item Observe that to compute the interpolation operator $\mathfrak{P}^{A_i}$, we only need to know values on $\mathbb{X}(A_i)$, thus the sequence $(A_i)$ will generally be given as
\begin{align}
A_{i+1}:= \left(\bigcup_{x \in \mathbb{X}(A_i) } \mathop{supp}(\hat{X}^{t_{i},x}_{t_{i+1}}) 
\right)^\blacksquare,\; 0\le i \le n-1,
\end{align}
where 
$ (H)^{\blacksquare}$ is the minimal hypercube that contains $H$, i.e.
\[ (H)^{\blacksquare} := \bigcap \mbrace{ A:  H \subset A , A = \prod_{i=1}^d [a_i,b_i] \text{ for } \bm{a},\bm{b} \in \R^d           } .\]
\end{enumerate}
\end{Remark}

\noindent With these notations, we can finally introduce
the projected-cubature backward approximation. For all $x\in A_i$, define

\begin{align}
\bar {u}_i (x) & =  \mathfrak{P}_{i}[\check{u}_i](x) \text{ with } \check{u}_i(x)= \E^{\hQ} \! \left[  \bar{u}_{i+1} \left( \hat{X}_{t_{i+1}}^{t_i,x} \right)   \right]   + h_i f(x,\check{u}_i(x), \check{v}_i(x))  \,,
\label{eq de pre proj scheme u}\\
\check {v}_i (x) & =   \E^{\hQ} \! \left[ \bar{u}_{i+1}\left(  \hat{X}_{t_{i+1}}^{t_i,x}  \right)  \frac{\Delta \hat{w}_i}{h_i}   \right]   \,,\label{eq de pre proj scheme v}
\end{align}

\noindent The terminal condition is set to $(\bar{u}_n,\bar{v}_n) = (g,0)$. 

\vspace{2mm} 
\noindent Thanks to the interpolation operator, to obtain the value of the approximation at time $0$, namely $\check{u}_0(x_0)$, we only need in practice to compute the above scheme for $x \in\mathbb{X}(A_i)$, $1 \le i \le n-1$: This is the main source of complexity for our algorithm. Let
\begin{align}
\cD := \bigcup_{1 \le i \le n-1}\set{t_i}\times \mathbb{X}(A_i)
\end{align}
we shall thus measure below the complexity of our methods in terms of $\sharp \cD$ the cardinal of $\cD$. However, let us  insist  on the fact that, our approximation is then defined and available, without loss of precision, on the bigger space $\bigcup_{i=0}^{n-1}\set{t_i}\times A_i$.

The complexity is obtained through a careful analysis of the scheme convergence and will be computed below for an example of linear interpolation operators. We conclude this section with a key proposition stated in this  abstract setting and which compares the projected-cubature backward approximation with the cubature scheme .

\begin{Proposition} \label{pr abstract control}
Assume that the $(\mathfrak{P}_i)$ satisfy \eqref{eq stability projection}-\eqref{eq consistency projection} with functions $(\eta_i,\epsilon_i)$, denote $N_i = \mathbb{X}(A_i)$, then the following holds
\begin{align}\label{eq pr abstract control}
\max_{0 \le i \le n}|\hat{u}_i - \bar{u}_i|_\infty \le Ce^{\sum_{i=0}^{n-1} \eta_i(N_i)} \sum_{i=0}^{n-1} \epsilon_{i}(N_{i})\vvvert \hat{u}_{i} \vvvert \;, 
\end{align}
where, by a slight abuse of notation, $|\cdot|_\infty$ denotes the sup-norm on $\cC^0(A_i)$, which reduces simply, for $i=0$, to $|\hat{u}_0(x_0) - \bar{u}_0(x_0)|$.
\end{Proposition}
\proof 1. We compare $(\hat{u},\hat{v})$ with $(\bar{u},\bar{v})$, which combines studying the stability of scheme of type \eqref{eq de pre proj scheme u}-\eqref{eq de pre proj scheme v} and some truncation error
given by
\begin{align*}
\varepsilon_i := (\mathfrak{I}-\mathfrak{P}_i)\hat{u}_i \;.
\end{align*}
recall the definition of $(\hat{u}_i,\hat{v}_i)_{0 \le i \le n}$ in Section \ref{subse de backward scheme}.
Observe that
\begin{align*}
|\hat{u}_i - \bar{u}_i|_\infty \le |(\mathfrak{I}-\mathfrak{P}_i)\hat{u}_i|_\infty + |\mathfrak{P}_i \hat{u}_i - \bar{u}_i|_\infty
\end{align*}
Now recalling that $\bar{u}_i = \mathfrak{P}_i \check{u}_{i}$ and using \eqref{eq stability projection}-\eqref{eq consistency projection}, we obtain
\begin{align}\label{eq temp control error}
|\hat{u}_i - \bar{u}_i|_\infty \le \epsilon_i (N_i) \vvvert \hat{u}_i \vvvert + e^{\eta(N_i)}|\hat{u}_i - \check{u}_i|_\infty \;.
\end{align}
The second term in the right hand side of the previous inequality is upper bounded as follows,
\begin{align}\label{eq one step statibility}
|\hat{u}_i - \check{u}_i|_\infty \le e^{C h_i}|\hat{u}_{i+1} - \bar{u}_{i+1}|_\infty\;.
\end{align}
This control is well known and has been obtained several times in slightly different contexts. For sake of completeness, we shall give a short proof below. Now, inserting back the previous inequality in \eqref{eq temp control error} and iterating on $i$, we obtain \eqref{eq pr abstract control}.
\\
2. We now prove \eqref{eq one step statibility}. 
Let us denote $\Delta \bar{u}_i = \hat{u}_i - \bar{u}_i$, $\Delta \check{u}_i = \hat{u}_i - \check{u}_i$ , $\Delta \check{v}_i = \hat{v}_i - \check{v}_i$ and $\delta f_i(x) = f(x,\hat{u}_i(x), \hat{v}_i(x))-f(x,\check{u}_i(x), \check{v}_i(x))$. We then have
\begin{align*}
\Delta \check{u}_i (x) =  \Delta \bar{u}_{i+1}\!\left( \hat{X}_{t_{i+1}}^{t_i,x} \right)  + h_i \delta f_i(x) - \Delta \check{v}_i(x)\Delta \omega_i - \Delta M_i(x) %
\end{align*}
where $\Delta M_i(x)$ satisfies
\begin{align*}
\E^{\hQ} \! \left[ \Delta M_i(x) \right] = \E^{\hQ} \! \left[ \Delta M_i(x) \Delta \omega_i\right] = 0 \,.
\end{align*}
Using the equality $|a|^2 = |b|^2 + 2b(a-b) + |a-b|^2$ with $b=\Delta \check{u}_i (x)$ and $a=\Delta \bar{u}_{i+1}\!\left( \hat{X}_{t_{i+1}}^{t_i,x} \right)$, we obtain, 
\begin{align}\label{eq discrete equation}
|\Delta \check{u}_i(x)|^2 &= \E^{\hQ} \left[|\Delta \bar{u}_{i+1}\!\left( \hat{X}_{t_{i+1}}^{t_i,x} \right)|^2 
+ 2h_i \Delta \check{u}_i (x) \delta f_i(x)
- |\Delta \bar{u}_{i+1}\!\left( \hat{X}_{t_{i+1}}^{t_i,x} \right)-\Delta \check{u}_i (x)|^2
\right]
\end{align}
Since
\begin{align*}
\Delta \check{v}_i (x) = \E^{\hQ} \left[ 
\left\{\Delta \bar{u}_{i+1}\!\left( \hat{X}_{t_{i+1}}^{t_i,x} \right) - \Delta \check{u}_i (x) \right\} \frac{\Delta \hat{w}_i}{h_i}   \right]
\end{align*}
we obtain using Cauchy-Schwarz inequality,
\begin{align*}
-\E^{\hQ} \left[ |
\Delta \bar{u}_{i+1}\!\left( \hat{X}_{t_{i+1}}^{t_i,x} \right) - \Delta \check{u}_i (x)  |^2  \right]
\le
-h_i|\Delta \check{v}_i (x)|^2 \;.
\end{align*}
Letting $L$ be the Lipschitz constant of $f$ and using Young's inequality, we also get
\begin{align*}
2 |\Delta \check{u}_i (x) \delta f_i(x)| \le (\alpha+2L)|\Delta \check{u}_i (x)|^2 + \frac{L^2}\alpha |\Delta \check{v}_i (x)|^2\;.
\end{align*}
for some $\alpha >0$ that will be set later on.
Combining the two previous inequalities with \eqref{eq discrete equation}, we compute
\begin{align*}
|\Delta \check{u}_i (x)|^2\set{1- (1+\alpha+2L)h_i }&\le \E^{\hQ} \left[|\Delta \bar{u}_{i+1}\!\left( \hat{X}_{t_{i+1}}^{t_i,x} \right)|^2 \right] +  h_i  {(\frac{L^2}\alpha-1)} |\Delta \check{v}_i (x)|^2.
\end{align*}
Setting $\alpha$ large enough, we obtain, for $h$ small enough,
\begin{align*}
|\Delta \check{u}_i (x)|^2 \le e^{Ch_i} \left(\E^{\hQ} \left[|\Delta \bar{u}_{i+1}\!\left( \hat{X}_{t_{i+1}}^{t_i,x} \right)|^2 \right] \right)
\end{align*}
for some $C>0$ that does not depend on $h$. Taking the supremum on first the right hand side and then on the left hand side, concludes the proof of \eqref{eq one step statibility}.
\eproof

\subsection{Example of multi-linear interpolation}
In this Section, we give an explicit specification of the projection operator and we study the complexity of two fully implementable methods: namely the basic Euler scheme and a second order method obtained through a Richardson-Romberg extrapolation. In order to simplify the presentation of the results, we use a uniform grid and we strengthen the assumption on the diffusion parameters.
We work then  assuming that there exists $\Lambda>0$ such that
\begin{align}\label{eq unif ellip setting}
 \frac{1}{\Lambda} |x|^2 \leq  x ^\dagger\sigma\sigma^\dagger x \leq \Lambda |x|^2 ; \qquad \text{for all } x\in \R^d. 
\end{align}
The sequence multilinear interpolation operator for $(\mathfrak{P}_i)$  is defined as follows. 
Recall that $A= \prod_{\ell=1}^d [a_\ell,b_\ell] $ is a hypercube in $\R^d$. We denote $|A| := \max_{\ell}|b_\ell- a_\ell |$.   Given a multi-index  $\bm{l} \in (\NN^*)^{d}$ we define a vector of grid sizes
\[ \bm{\delta}_{\bm{l}}(A):= ( \delta_{l_1},\ldots , \delta_{l_d}   ) ; \quad \text{where }
 \delta_{l_\ell} = \frac{ b_\ell-a_\ell}{l_\ell}\text{ for } \ell=1,\ldots d.\]
This set of distances defines a grid $\mathbb{X}(A)$ with nodes denoted by
\[ \bm{\check{x}}_{\bm{l},\bm{j}}(A):= ( a_1+ j_1\delta_{l_1}, \ldots , a_d+ j_d\delta_{l_d}  ), \qquad \text{for } \bm{0}\leq \bm{j} \leq \bm{l}.   \]
We denote $\check{x}^\ell_{l_\ell,j_\ell} := a_\ell+ j_\ell \delta_{l_\ell}$, $1\le \ell \le d$, and observe that $\bm{\check{x}}_{\bm{l},\bm{j}}(A) = (\check{x}^\ell_{l_\ell,j_\ell})_{1\le \ell\le d}$.
By setting
\[ \phi(x):= \begin{cases} 1-|x| & \text{if } x\in [-1,1]\\ 0 & \text{otherwise} \end{cases}, \]
we can define an associated set of nodal basis functions given by
\begin{equation}
 \phi_{\bm{l},\bm{j}} ( x; A ) := \prod_{\ell=1}^d \phi \mpar{ \frac{x_\ell -\check{x}^\ell_{l_\ell,j_\ell} }{\delta_{l_\ell} } } .
 \label{Eq:DefNodalBasis}
\end{equation}
In practice, we use a grid with the same number of points in each direction. 
\begin{align}\label{eq complexity grid t_i}
N_i := \sharp\mathbb{X}(A_i) = l_i^d\;.
\end{align}
and the interpolation operator is  given by
\begin{align}\label{eq de linear interpolation}
\mathfrak{P}_i \varphi(\cdot) = \sum_{\bm{0}\leq \bm{j} \leq \bm{l}} \varphi(\bm{\check{x}}_{\bm{l}_i,\bm{j}}(A_i))\phi_{\bm{l}_i,\bm{j}} (\cdot,A_i)\,,\quad \text{ for all } \varphi \in \cC_0(A_i)\;.
\end{align}
It remains to precise how the $A_i$ are chosen: 
The hypercube $(A_i)$ are defined in a ``minimal'' way according to Remark \ref{re const A_i}. 
\\
For this sequence of linear interpolation operator, $(\mathfrak{P}_i,\mathscr{R}^a,\mathscr{R}^s)$ is  {compatible} in the sense \eqref{eq stability projection}-\eqref{eq consistency projection} with $\mathscr{R}^a = C^2(A_i)$ and  $\mathscr{R}^s = C^0(A_i)$.  
The functions $(\epsilon,\eta)$ are also well known, see e.g. \cite{bungartz_sparse_2004} and we recall them in the following Lemma. 
\begin{Lemma}\label{le consistency - stability interpolation} For $\mathfrak{P}_i$, the multi-linear operator defined above, the following holds
\begin{enumerate}[i)]
\item For all $i$, $\eta_i \equiv 0$. 
\item Set $N_i := \sharp \mathbb{X}(A_i)$, for all $\varphi \in \cC_2(A_i)$
\begin{align*}
|(\mathfrak{I}-\mathfrak{P}_i)\varphi|_\infty \le  \epsilon_i(N_i) \| \varphi \|_{2,\infty}
\; \text{ with } \; \epsilon_i(N) := c \frac{|A_i|^2}{N^{\frac2d}} \;,
\end{align*}
for some $c>0$ which does not depend on $|A_i|$ nor $\varphi$.
\end{enumerate}
\end{Lemma}

\begin{Lemma} \label{le bound c2-norm scheme} Assume that Assumption \ref{H0} is in force with $M\geq 3$, then the following property holds for the backward scheme of Section \ref{subse de backward scheme}:
\begin{align*}
\sup_i \| \hat{u}_i \|_{2,\infty} \le C\;.
\end{align*}
\end{Lemma}
\proof
The claim follows directly from differentiation under the conditional expectation and the boundedness of $f,g$ and their derivatives. \qed

For the sake of clarity, we will from now on indicate in superscript the number of time steps of the scheme under consideration. This will prove useful as we will introduce an extrapolation method.

The following results demonstrate the usefulness of the expansion result of Theorem \ref{Thm:MainExpansionBwd} when combined with the multi-linear interpolation procedure to reduce the complexity of the cubature method. Moreover, to profit from the error expansion, we introduce a Richardson-Romberg extrapolation and define:
\begin{align}
\overline{\overline{u}}^n_0(x_0) := 2\overline{u}^{2n}_0(x_0) - \overline{u}^n_0(x_0)\;, \; n \ge 1.
\end{align}
We then have the following results
\begin{Theorem}\label{th linear projection}
 Let $\epsilon > 0$ be a given precision. Then, the following holds
\begin{enumerate}
\item Order one method: setting $n \sim \frac1{\epsilon}$ and $N_i \sim n^{\frac{d}2}i^d$, for $i \le n-1$, we have that the complexity $\mathscr{\cC} := \sharp \cD$ satisfies
\begin{align}
\cC(\epsilon) = O(\epsilon^{-\frac{3d+1}2}) \;\text{ for }\; |u(0,x_0)-\bar{u}^n_0(x_0)| = O(\epsilon)\;.
\end{align}
\item Order two method: setting $n \sim \frac1{\sqrt{\epsilon}}$ and $N_i \sim n^d i^d$, for $i \le n-1$, we have that the complexity $\mathscr{\cC} := \sharp \cD$ satisfies
\begin{align}
\cC(\epsilon) = O(\epsilon^{-\frac{2d+1}2}) \;\text{ for }\;  |u(0,x_0)-\overline{\overline{u}}^n_0(x_0)| = O(\epsilon)\;.
\end{align}

\end{enumerate}
\end{Theorem}

\proof
1. Order 1 method: We first identify the convergence property thanks to the previous sections
\begin{align}\label{eq order 1 starting point}
|u(0,x_0) - \bar{u}^n_0(x_0)|
\le
|u(0,x_0) - \hat{u}^n_0(x_0)| + \max_i |\hat{u}^n_i - \bar{u}^n_i|_\infty \;.
\end{align}
Combining Proposition \ref{pr abstract control} with Lemma \ref{le consistency - stability interpolation} and Lemma \ref{le bound c2-norm scheme}, we get
\begin{align*}
\max_i |\hat{u}^n_i - \check{u}^n_i|_\infty \le C \sum_{i=1}^{n-1} \left( \frac{|A_i|}{N_i^{\frac1d}} \right)^2
\end{align*}
The uniform-ellipticity assumption on the diffusion
coefficients \eqref{eq unif ellip setting}, and matrix norm equivalence implies that there exist two
constants, \(\lambda_{-}, \lambda_{+}\) such that
\[ \lambda_{-} \sum_{k=0}^i h_k^{1/2} \leq |A_i| \leq \lambda_{+} \sum_{k=0}^i h_k^{1/2}. \]
For the regular time grid we thus get
\begin{align*}
\max_i |\hat{u}^n_i - \check{u}^n_i|_\infty \le \frac{C}{n} \sum_{i=1}^{n-1} \left( \frac{i }{(N^n_i)^{\frac1d}} \right)^2
\end{align*}
Combining \eqref{eq order 1 starting point} with  Theorem \ref{Thm:MainExpansionBwd} and the previous inequality yields \begin{align*}
|u(0,x_0) - \bar{u}^n_0(x_0)| \le \frac{C}{n} \left( 1 + \sum_{i=1}^{n-1} \left( \frac{i }{(N^n_i)^{\frac1d}} \right)^2 \right)
\end{align*}
Setting $N_i = n^{\frac{d}2}i^d$, we compute  $|u(0,x_0) - \bar{u}^n_0(x_0)|  \le \frac{C}n$
and the overall complexity is given
\begin{align*}
\cD = \sum_{i=1}^{n-1} N^n_i = O(n^{\frac{3d+1}2}).
\end{align*}
To reach a precision $\epsilon$, we thus obtain $\mathscr{C}(\epsilon)=O(\epsilon^{-\frac{3d+1}2})$.
\\
2. Order 2 method: We first observe that
\begin{align}\label{eq o2 starting point}
|u(0,x_0) - \overline{\overline{u}}^n_0(x_0)|
\le & |u(0,x_0) - \left(2 \hat{u}^{2n}_0(x_0) -\hat{u}^{n}_0(x_0) \right)|
\\
&
+ 2|\hat{u}^{2n}_0(x_0)-\bar{u}^{2n}_0(x_0)| + |\hat{u}^{n}_0(x_0)-\bar{u}^{n}_0(x_0)|
\end{align}
Let $N_i^{2n} = \sharp A_i^{2n}$ (resp. $N_i^{n} = \sharp A_i^{n}$) for $i \le 2n-1$ (resp. $i \le n-1$ ) be the number of points in each grid at each time-step $t_i$ for the method with $2n$ time-steps (resp. $n$ time-steps). 
Following the computation of the previous step, we obtain
\begin{align*}
2|\hat{u}^{2n}_0(x_0)-\bar{u}^{2n}_0(x_0)| + |\hat{u}^{n}_0(x_0)-\bar{u}^{n}_0(x_0)|
\le
C \left( \sum_{i=1}^{2n-1}\left( \frac{i}{(N^{2n}_i)^\frac1{d}}\right)^2 + \sum_{i=1}^{n-1} \left(  \frac{i}{(N^{n}_i)^\frac1{d}}\right)^2 \right)\;.
\end{align*}
Now, setting $N^n_i \sim N_i^{2n} \sim (ni)^d $, we compute
\begin{align*}
2|\hat{u}^{2n}_0(x_0)-\bar{u}^{2n}_0(x_0)| + |\hat{u}^{n}_0(x_0)-\bar{u}^{n}_0(x_0)|
\le
\frac{C}{n^2}\;.
\end{align*}
Combining the previous inequality with \eqref{eq o2 starting point} and the order two expansion of Theorem \ref{Thm:MainExpansionBwd}, we get
\begin{align}
|u(0,x_0) - \overline{\overline{u}}^n_0(x_0)| \le \frac{C}{n^2}\;.
\end{align}
We now observe that the overall complexity is given
\begin{align*}
\cD = \sum_{i=1}^{n-1} N^n_i + \sum_{i=1}^{2n-1} N^{2n}_i  = O(n^{2d+1}).
\end{align*}
Of course, the gain comes from the precision that is obtained and setting $\epsilon \sim \frac1{n^2}$, we finally compute that
\begin{align*}
\cC(\epsilon) = O(\epsilon^{-d-\frac12})
\end{align*}
\\
\\

\eproof

\subsection{Numerical illustration}
\label{Sec:NumericalTest}

We test numerically the efficiency of our approximation scheme and, in particular, the gain in complexity coming from the Romberg-Richardson method.  We consider the following system:

 \begin{itemize}
 \item Forward equation:  $X= W $, a d-dimensional Brownian Motion.
 \item Backward equation: 
 \begin{align*}
 Y_t = g(W_1) + \int_t^1 f(Y_s,Z_s) \ud s - \int_t^1 Z_s \ud W_s
 \end{align*}
 with $f: \R\times \R^d \rightarrow \R$ given by
\[f (y,z)= \left(y-\frac{2+d}{2d}\right) \sum_{\ell=1}^d z_\ell\]
 and 
 $$g(x)=\frac{k(1,x)}{1+k(1,x)}, \; \text{ with } k(t,x) =  \exp\left(t+\sum_{\ell=1}^d x_\ell\right). $$
\end{itemize}
We use a cubature formula in dimension $d$ of order $m=3$. It is defined, for $j=1, \ldots, 2d$ by $p_j = (2d)^{-1}$ and $\omega_j = (-1)^j e_{\lceil j/2 \rceil} $, where $e_i$ is the $i$th canonical basis element (see \cite{gyurko_efficient_2011}).

\noindent The previous system can be solved analytically. In particular, we get 
$$Y_0= \frac{k(0, \mathbf{0})}{1+k(0,\mathbf{0})} = \frac12.$$

\noindent  {Although the function $f$ is not globally Lipschitz, it can be treated as such, given that it is locally Lipschitz and the processes $Y,Z$ are bounded in the application. Hence, the main results apply in this setting.}

\noindent We look at the results of the algorithm and its extrapolated version. 

 {
We use a sparse implementation of the linear interpolation operator introduced above. We refer the reader to the Appendix for a quick presentation of the sparse grid setting and refer to the seminal paper \cite{bungartz_sparse_2004} for more insight on this topic. Note that, the use of \emph{sparse grid} has already been suggested in the context of BSDEs approximation in \cite{zhang_sparse-grid_2013}, but the forward approximation method used in this paper is different. The sparse grid implementation allows  to obtain numerical results with almost same precision as the linear interpolation but in a smaller running time. 
}
The rate of convergence of the full method is shown in Figure \ref{Fig:NumTests} (Left). 
The original Euler algorithm shows the expected rate of convergence. 
The extrapolated one converges even faster, showing there is possibly an extra cancellation for the next order term. 
In Figure \ref{Fig:NumTests} (Right),  an illustration of the rate of convergence and complexity of the scheme in terms of the time complexity of the algorithm is shown. 
We can see that there is an effective reduction on the overall time to solve the problem with a given error.

 {
Unfortunately, the use of the sparse grid approximation is out of the scope of the theoretical results stated in Section \ref{subse interpolator}. We were not able to obtain the stability property \eqref{eq stability projection} for the sparse interpolator.  Nevertheless, for our numerical example, the stability seems to hold true in practice. It seems to be a challenging question to understand the conditions under which this property could be true.
}

\begin{figure}[H]
\includegraphics[width=0.45\textwidth]{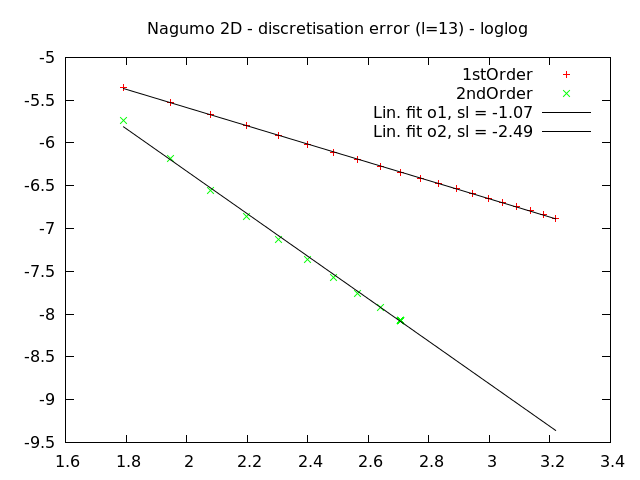}
\includegraphics[width=0.45\textwidth]{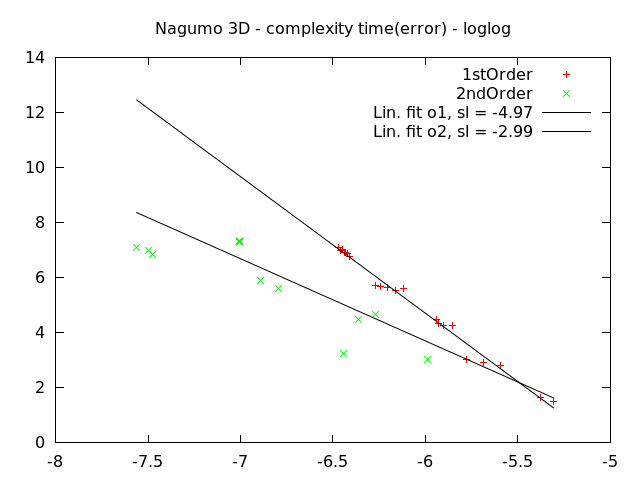}
\caption{Left: Log of the error of the scheme in terms of the log of the number of steps (n). 
The expected rate of convergence is observed for the original algorithm and a better than  expected for the extrapolated one. (Example in dimension 2). 
Right: Log-complexity time (in seconds) as a function of log error (Example in dimension 3). 
Data as in the text.}
\label{Fig:NumTests}
\end{figure}

\appendix
\section{Auxiliary results on the decreasing step discretization}

\begin{Proposition}
Let $a>1$, $n>0$. There exists $C>0$ such that 
\[\frac{1}{n^a} \sum_{k=1}^n \left(\frac{k}{n}\right)^b \leq \begin{cases}  C n^{1-a} & \text{ if } b>-1 \\ C n^{1-a}\log(n) &\text{ if } b=-1 \\ C n^{-(a+b)} &\text{ if } b<-1    \end{cases} .\]
\label{Prop:ControlSumAbs}

\end{Proposition}

\begin{proof}

\begin{itemize}
\item If $b \geq 0$,  the function $x^b$ is non-decreasing in $[0,1]$ and 
\[\frac{1}{n}  \sum_{k=1}^n \left(\frac{k}{n}\right)^b \leq \int_{1/n}^1 x^b \ud x = \frac{1}{b+1}(1-n^{-(b+1)}) \leq (b+1)^{-1}.\]
\item If $-1 \leq b < 0$, the function $x^b$ is decreasing in $[0,1]$ and 
\[\frac{1}{n}  \sum_{k=1}^n \left(\frac{k}{n}\right)^b =  n^{-(b+1)} +   \frac{1}{n}  \sum_{k=2}^{n} \left(\frac{k}{n}\right)^b \leq 1+\int_{1/n}^{1} x^b \ud x,\]
where the last integral is bounded by $(b+1)^{-1}$ if $b>-1$ and by $\log(n)$ if $b=-1$. 
\item If $b<-1$, the series $\sum_{k=1}^n k^{b}$ is convergent and increasing. The claim follows with $C$ the limit of the series. 
\end{itemize}
\end{proof}

\begin{Lemma}
Let $\gamma \geq 1, \ell \geq 1$. Let $\psi$ be such that in $[0,T)$, 
\[  |\psi(t)| \leq C |T-t|^{-\beta}   \]
Then, there exists a  $C'>0$ such that
\begin{equation} 
\sum_{k=0}^{n-1} |\psi(t_k)| h_k^\ell   \leq  
\begin{cases}
C'n^{-(\ell-1)} & \text{if } \gamma(\ell-\beta)>\ell-1\\
C'n^{-(\ell-1)}  \log(n) & \text{if } \gamma(\ell-\beta)=\ell-1\\
C' n^{- \gamma ( \ell-\beta) } &\text{if }  \gamma(\ell-\beta)< \ell-1  
\end{cases}.
\label{Eq:ControlSumPsi}
\end{equation}
\label{Lem:ControlSum}
\end{Lemma}

\proof
Remark that
\[h_{k} = T \gamma \int_{\frac{k}{n}}^{ \frac{k+1}{n}} (1-x)^{\gamma-1} \ud x   \]
and  for $\gamma \geq 1$, $(1-x)^{\gamma -1}$ is decreasing in $[0,1]$. This implies
\begin{equation}  
\frac{T\gamma}{n} \mpar{1-\frac{k+1}{n}}^{\gamma-1} \leq  h_k \leq  \frac{T\gamma}{n}  \mpar{1-\frac{k}{n}}^{\gamma-1} .
\label{Eq:ControlDeltaK}
 \end{equation}
 
Thus, we obtain
\begin{align*}  
\sum_{k=0}^{n-1} |\psi(t_k)|   h_k ^{\ell} \leq  &  \mpar{ \frac{T\gamma}{n}}^{\ell}  \sum_{k=0}^{n-1} C |T-t_k|^{-\beta}  \mpar{ 1- \frac{k}{n} }^{\ell(\gamma-1) }  \\
= &   C T^{\beta} \mpar{ \frac{ T \gamma}{n}}^{\ell}  \sum_{k=0}^{n-1}   \mpar{ 1- \frac{k}{n} }^{\ell(\gamma-1)-\beta\gamma }  \\
= &  C T^{\beta} \mpar{ \frac{T\gamma}{n}}^{\ell}  \sum_{k=1}^{n}   \mpar{ \frac{k}{n} }^{\ell(\gamma-1)-\beta\gamma }.
\end{align*}
We conclude from Proposition \ref{Prop:ControlSumAbs}.
\eproof

\begin{Corollary}
With the assumptions of Lemma \ref{Lem:ControlSum},
\[
\sum_{k=0}^{n-2} \sup_{t\in [t_k,t_{k+1}]} |\psi(t)| h_k^\ell   \leq  
\begin{cases}
C'n^{-(\ell-1)} & \text{if } \gamma(\ell-\beta)>\ell-1\\
C'n^{-(\ell-1)}  \log(n) & \text{if } \gamma(\ell-\beta)=\ell-1\\
C' n^{- \gamma ( \ell-\beta) } &\text{if }  \gamma(\ell-\beta)< \ell-1  
\end{cases}.
\]
\label{Cor:ControlSum2}
\end{Corollary}

\proof 
 {
Remark that 
\[\sup_{t\in [t_k,t_{k+1}]} |\psi(t)|  \leq \sup_{t\in [t_k,t_{k+1}]} C |T-t|^{-\beta}.\]} If $\beta\leq 0$, the extreme on the right hand side above is attained at $t_k$, and we proceed exactly as in Lemma \ref{Lem:ControlSum}. Otherwise, the extreme is attained at $t_{k+1}$ and
 \begin{align*}
 \sum_{k=0}^{n-2} \sup_{t\in [t_k,t_{k+1}]} |\psi(t)| h_k^\ell 
 & \leq  C T^\beta\sum_{k=0}^{n-2} (1-\frac{k+1}n)^{-\gamma \beta} h_k^\ell
 \\
 &\leq  C T^\beta\mpar{ \frac{T\gamma}{n}}^{\ell}  \sum_{k=0}^{n-2} (1-\frac{k+1}n)^{-\gamma\beta}  \mpar{ 1- \frac{k}{n} }^{\ell(\gamma-1) }
 \\
 &\leq C T^\beta\mpar{ \frac{T\gamma}{n}}^{\ell} 2^{\ell(\gamma-1)} \sum_{k=1}^{n-1} \mpar{\frac{k}n}^{-\gamma\beta+\ell(\gamma-1)}  
 \end{align*}
and we conclude from Lemma \ref{Lem:ControlSum}.
\eproof

\begin{Lemma}
Under the same assumptions of Lemma \ref{Lem:ControlSum},  suppose that $\gamma(\ell-\beta)>\ell$ and that $\psi$ is of bounded variation in $[0,T-\epsilon]$ for all $\epsilon>0$. Then, for all $i<n-1$
\begin{equation}  
\sum_{k=i}^{n-1} \psi(t_k) h^\ell_{k} =    \mpar{ \frac{T\gamma}{n} }^{\ell-1} \int_{t_i}^{T} \psi(t) \mpar{ 1- \frac t T }^{ \beta^* }   \ud t  + R(i).
\label{Eq:ControlSumPsiIntegral}
\end{equation}
where
\begin{equation} 
\beta^* := (\ell-1)(1 - \gamma^{-1}).
\label{Eq:DefBetaStar}
\end{equation}

and for some $K>0$,
 {
\[|R(i)| \leq  K\mpar{ \frac1{n^{\gamma(\ell-\beta)+1}} + \frac{1}{n^\ell} \mpar{1-\frac{i-1}{n}}^{\gamma(\ell-\beta)-\ell} }    \]
}
\label{Lem:ApproxSum}
\end{Lemma}

\proof
First, we observe that
\begin{align*}
\int_{t_i}^T\psi(t)\left(1-\frac{t}T\right)^{\beta^*} \ud t = T\gamma \int_{i/n}^1 \psi(T\set{1-(1-x)^\gamma})(1-x)^{\gamma (\beta^*+1)-1} \ud  x =:\int_{i/n}^1 \theta(x) \ud x\;,
\end{align*}
and since $\gamma(\ell-\beta)>\ell$
\begin{align}\label{eq growth theta}
|\theta(x)| \le C(1-x)^{\gamma(\beta^*-\beta+1) -1} = C(1-x)^{\gamma(\ell-\beta) -\ell}\;.
\end{align}
We compute that
\begin{align*}
\sum_{k=i}^{n-1}\psi(t_k)h_k^\ell = \frac{(T \gamma)^\ell}{n^{\ell-1}} \int_{i/n}^1\theta(x) \ud x + R^2_n+R^1_n\;,
\end{align*}
with
\begin{align*}
R^1_n(i) &:= \left(\frac{T \gamma}{n}\right)^{\ell} 
\sum_{k=i}^{n-1}\psi(t_k)\left( \left(\frac{n}{T \gamma}\right)^{\ell} h^\ell_k - \left(1-\frac{k}n\right)^{\ell(\gamma-1)}\right) \;,
\\
R^2_n(i) &:=\frac{(T \gamma)^\ell}{n^{\ell-1}} \left( \sum_{k=i}^{n-1}\frac{\psi(t_k)}{n}\left(1-\frac{k}n\right)^{\ell(\gamma-1)}  - \int_{i/n}^1 \theta(x) \ud  x \right)\;.
\end{align*}
We now study each remainder term separately.
\\
\begin{enumerate}[i.]
\item  We observe that, from \eqref{Eq:ControlDeltaK}, we have
\begin{align}\label{eq control}
\bigg| \left(1-\frac{k}{n}\right)^{\ell(\gamma-1)} - \left(\frac{n}{T\gamma}\right)^\ell h_k^\ell \bigg|
& \le\bigg| \left(1-\frac{k}{n}\right)^{\ell(\gamma-1)} - \left(1-\frac{k+1}{n}\right)^{\ell(\gamma-1)}\bigg| \nonumber
\\
& \le \frac{\ell(\gamma-1)}{n}\left(1-\frac{k}{n}\right)^{\ell(\gamma-1)-1}\;,
\end{align}
since $\ell(\gamma-1) \ge 1$.
\\

Using \eqref{eq growth theta} and \eqref{eq control} 
\begin{align*}
|R^1_n| &\leq \left(\frac{T \gamma}{n}\right)^{\ell} \frac{\ell\left(\gamma-1\right)}{n} \sum_{k=i}^{n-1} \left|\theta\left(\frac{k}{n}\right)\right|\left(1-\frac{k}n\right)^{-1}
\\
& \leq \left(\frac{T \gamma}{n}\right)^{\ell} \frac{\ell(\gamma-1)}{n} \sum_{k=1}^{n-i} \left(\frac{k}n\right)^{\gamma(\ell-\beta) -\ell-1}\\
&\le  {\frac{K}{n^\ell} \mpar{1-\frac{i-1}{n}}^{\gamma(\ell-\beta)-\ell}   }\;,
\end{align*}
where we used, for the last inequality Proposition \ref{Prop:ControlSumAbs}, and the fact that $\gamma(\ell-\beta)> \ell$.

\vspace{5pt}
\item For $R^2_n$, we assume that $\theta$ is increasing. For $\theta$ decreasing, similar computations as the one below can be made and the result holds true for $\theta$ as it has bounded variation on $[0,t_{n-1}]$. We thus observe that, for $k \ge 1$,
\begin{align*}
\theta\left(\frac{k-1}{n}\right) \leq n \int_{\frac{k-1}n}^\frac{k}{n} \theta(x) \ud  x \leq \theta\left(\frac{k}n\right) .
\end{align*}
Summing the previous inequalities and rearranging terms, we obtain
\begin{align*}
\bigg | \int_{i/n}^1 \theta(x) \ud  x - \frac1n\sum_{k=i}^{n-1} \theta\left(\frac{k}n\right)\bigg|
\le \int_{1-\frac1n}^1|\theta(x)| \ud  x + \frac{|\theta(i/n)|}{n} + \frac1n \left|\theta\left(1-\frac1n\right)\right|\;.
\end{align*}
We then compute
\begin{align*}
\frac1n \left|\theta\left(1-\frac1n\right)\right| \le K\frac1{n^{\gamma(\beta^*-\beta+1) +1}}\;,
\end{align*}
and
\begin{align*}
\int_{1-\frac1n}^1|\theta(x)| \ud  x  \le C \int_{1-\frac1n}^1(1-x)^{\gamma(\beta^*-\beta+1) -1}\ud x\le K\frac1{n^{\gamma(\beta^*-\beta+1) +1}}\;,
\end{align*}
 {Finally, \eqref{eq growth theta} controls the remaining term}. This concludes the proof.
\end{enumerate}
\eproof

\section{Sparse grid implementation}

We introduce here the numerical method that has been used in practice. This is a sparse version of the linear interpolation operator presented in the theoretical analysis of Section \ref{se complexity}. The idea is to use less nodal functions than in the case of a full linear interpolation with a minimal degradation of the error but a great improvement in complexity (number of points needed for the interpolation). This sparse grid concept is thoroughly reviewed in the seminal paper \cite{bungartz_sparse_2004}.

Let us thus consider the sparse grid nodal space of order $p$ defined by
\[\mc{V}_p (A):= \mathrm{span} \mbrace{ \phi_{\bm{l},\bm{j}}  ;  (\bm{l},\bm{j}) \in \mc{I}_p (A)   },\]
where
\begin{align}
\mc{I}_p (A):=  \big\{ (\bm{l},\bm{j})  : & \quad 0\leq \sum_{i=1}^d l_i \leq p; \quad  \bm{0}\leq \bm{j} \leq \bm{2^l}; \notag \\ & ( l_i>0 \text { and }  j_i \text{ is odd})   \text{ or }  ( l_i=0 ),  \text{ for } i=1,\ldots, d \big\}  \label{Eq:DefIp}.  
\end{align}

For a function $\psi:A \rightarrow \R$ with support in $A$, we define its associated $\mc{V}_p$-interpolator by
\begin{align}  \label{eq sparse ope}
\sprse{p}{A}{\psi}{x} :=\sum_{ (\bm{l},\bm{j}) \in \mc{I}_p (A) } \theta_{\bm{l},\bm{j}}(\psi;A) \phi_{\bm{l},\bm{j}}(x; A) 
\end{align}
where the operator $\theta_{\bm{l,j}}$ can be defined recursively in terms of $r$, the dimension of $\bm{l}$, by:
\begin{equation} \theta_{\bm{l},\bm{j}} (\psi;A) = 
  \begin{cases}  
  \psi(\check{x}_{\bm{l},\bm{j}}); &  r=0\\
  \theta_{\bm{l}-, \bm{j}-}  (\psi ( \cdot, \check{x}^r_{l_r,j_r} ) ;A-)    ;  & l_r =0\\
     \theta_{\bm{l}-, \bm{j}-}(\psi( \cdot, \check{x}^r_{l_r,j_r} );A-) -\frac{1}{2} \theta_{\bm{l}-, \bm{j}-}(\psi( \cdot, \check{x}^r_{l_r,j_r-1} );A-)   \\
     \qquad\qquad\qquad\qquad\qquad\  -\frac{1}{2} \theta_{\bm{l}-, \bm{j}-}(\psi(\cdot, \check{x}^r_{l_r,j_r+1} )   ;A-) ; &  l_r >0
  \end{cases} 
  \label{Eq:DefTheta}
\end{equation}  
 where,  for a hypercube $A= \prod_{i=1}^d [a_i,b_i] $, $A- := \prod_{i=1}^{d-1} [a_i,b_i]$ and for a multi-index $\bm{k}$ with dimension $r \ge 1$, $\bm{k}-=(k_1,\dots,k_{r-1})$.   {The above definition has to be compared with the full linear interpolation operator given in \eqref{eq de linear interpolation}}.

With the sparse representation at hand, we modify the backward scheme and introduce first the set of points where the function needs to be approximated. Let
\begin{align}\label{eq de Di}
D_{i}  = \begin{cases}  \{x_0\} & \text{ if } $i=0$\\  \mpar{  \bigcup_{x\in D_{i-1}}  \supp \mbra{\hat{X}^{t_{i-1}, x}_{t_{i}}} }^{\mc{V}}_{p _i}& \text{ if } $i=1,\ldots,n-1$ \end{cases} , 
\end{align}
so that $D_i$ is a set of points in a sparse grid of order $p_i$. This sparse grid is constructed on the minimal hypercube that contains the diffusion started at $D_{i-1}$ at time $t_{i-1}$. We denote by $\cD$ the union of all the points in the grids $D_i$, namely 
\begin{align}\label{eq de all the points}
\cD := \bigcup_{i=0}^{n-1} \set{t_i}\times D_i
\end{align}
which forms  a finite grid of $[0,T)\times\R^d$.

Having this set, for a sequence of values $(p_1, \ldots, p_{n-1})$ with $p_i > d-1$, we define the sparse-cubature backward approximation by
 {
\begin{align}
\check {u}_i (x) & =  \sprse{p_i-d+1}{D_i}{ \left( \E^{\hQ} \left[   \check{u}_{i+1} \left( \hat{X}_{t_{i+1}}^{t_i,.} \right)   \right] \right)  + h_i f(\cdot,\check{u}_i(\cdot), \check{v}_i(\cdot)) }{x} 
\label{eq de full scheme u}\\
\check {v}_i (x) & =   
\left(  \E^{\hQ} \left[ \check{u}_{i+1} \left(  \hat{X}_{t_{i+1}}^{t_i,x}  \right)  \frac{\Delta \hat{w}_i}{h_i}   \right] \right)  
\label{eq de full scheme v}
\end{align}
}
for $x \in D_i^\blacksquare$. The terminal condition is set to $(\check{u}_n,\check{v}_n) = (g,0)$.

\vspace{4pt} In practice, the computational effort to obtain  $(\check{u}_i,\check{v}_i)_{i}$ is proportional to the number of points in the grid $\cD$. But, let us  insist  on the fact that, our approximation is then defined and available, without loss of precision, on the bigger space $\cup_{i=0}^{n-1}\set{t_i}\times D_i^\blacksquare$.

\bibliographystyle{plain}
\bibliography{BSDE_Error_Exp.bib}

\end{document}

%% file: basecom.tex
\newcommand{\cC}{\mathcal{C}}
\newcommand{\cD}{\mathcal{D}}

\newcommand{\cF}{\mathcal{F}}

\newcommand{\cL}{\mathcal{L}}

\newcommand{\E}{\mathbb{E}}
\newcommand{\F}{\mathbb{F}}
\renewcommand{\P}{\mathbb{P}}

\newcommand{\Q}{\mathbb{Q}}
\newcommand{\R}{\mathbb{R}}
\newcommand{\T}{\mathbb{T}}

\def \proof{{\noindent \bf Proof. }}
\def \eproof{\hbox{ }\hfill$\Box$}

\newcommand{\ud}{\mathrm{d}}

\newcommand{\set}[1]
    {\ensuremath{\{ #1 \}}}
    
\newcommand{\HP}[1] 
    {\ensuremath{\mathscr{H}^{#1}}}

\newcommand{\esp}[1]{\ensuremath{\mathbb{E} \!\! \left[#1\right] }}

\newcommand{\ti}[1]{t_{i #1}}

\renewcommand{\Xi}[1]{X_{i #1}}

\newcommand{\tv}{\tilde{v}}
\newcommand{\brv}{\breve{v}}